\documentclass[reqno,12pt,letterpaper]{amsart}
\usepackage[proof]{sdmacros}
\usepackage{enumerate}
\usepackage{comment}
\usepackage{subfigure}
\usepackage{amsmath,amsthm}
\usepackage{amssymb}
\newcommand{\RR}{{\mathbb R}}

\newcommand{\xR}{\mathbb R}
\newcommand{\NN}{{\mathbb N}}
\newcommand{\CC}{{\mathbb C}}

\DeclareMathOperator{\RE}{Re}

\DeclareMathOperator{\IM}{Im}

\def\eps{\varepsilon}

\def\la{\left\lvert}
\def\lA{\left\lVert}
\def\ra{\right\rvert}
\def\rA{\right\rVert}

\def\les{\lesssim}

\def\ba{\begin{align}}
\def\bad{\begin{aligned}}
\def\be{\begin{equation}}
\def\ea{\end{align}}
\def\ead{\end{aligned}}
\def\ee{\end{equation}}
\def\e{\eqref}

\def\Pr{\mathcal{P}}
\def\xT{\mathbb{T}}

\def\xR{\mathbb{R}}
\title{Damping for fractional wave equations and applications to water waves}
\author{Thomas Alazard}
\email{thomas.alazard@ens-paris-saclay.fr}
\address{Universit\'e Paris-Saclay, ENS Paris-Saclay, CNRS, Centre Borelli UMR9010, avenue des Sciences, F-91190 Gif-sur-Yvette France}
\author{Jeremy L. Marzuola}
\email{marzuola@math.unc.edu}
\address{Department of Mathematics, University of North Carolina, Chapel Hill, NC 27514}
\author{Jian Wang}
\email{wangjian@email.unc.edu}
\address{Department of Mathematics, University of North Carolina, Chapel Hill, NC 27514}

\begin{document}

\begin{abstract}
Motivated by numerically modeling surface waves for inviscid 
Euler equations, we analyze linear models for damped water waves and establish decay properties for the energy for sufficiently regular initial configurations.  Our findings give the explicit decay rates for the energy, but do not address reflection/transmission of waves at the interface of the damping. Still for a subset of the models considered, this represents the first result proving the decay of the energy of the surface wave models.    
\end{abstract}

\maketitle

\section{Introduction}
\label{sec: intro}

Motivated by highly successful numerical methods for damping the surface water wave equations proposed in the work \cite{clamond2005efficient}, we wish to establish a theory of absorbing boundary conditions/perfectly matched layers as an approach to a damped linear water wave models.  Such methods are essential to ensure that one can numerically simulate long-time behaviors of wave-trains without boundary interference.  While nonlinear damping mechanisms have been proposed using nonlinear properties of the water wave models in the work \cite{alazard2017stabilization,alazard2018stabilization,alazard2018control}, implementation of such methods can be numerically very stiff since the nonlinear damping mechanisms involve many spatial derivatives of the underlying models.  However, the proposed methods in the work \cite{clamond2005efficient} are extremely non-stiff, which we argue is strongly related to them arising due to linear mechanisms for damping.  Such connections should be explored further to fully understand the efficacy of these existing methods. The connection between damped waves and absorbing boundary conditions has long been understood for models with local differential operators, see for instance \cite{johnson2021notes,nataf2013absorbing} and references therein, so we endeavor here to extend our understanding of damping effects the inherently non-local models that arise the water wave problem.

Let us very briefly recall a Hamiltonian formulation of the evolution of a fluid interface in the gravity-capillary water wave system subject to an external pressure.  This equations can be written in terms of the surface height, denoted $\eta(x,t)$, and the velocity potential of the fluid restricted to the surface, denoted $\phi(x,t)$.  In \cite{Mar_Prior30}, the authors derive a robust method for numerically solving the Euler equations in very general geometric setting using the coordinate equations
\begin{align*}
  \eta_t = G(\eta) \phi, \ \ \phi_t = (\phi_\alpha/s_{0,\alpha})V + (\partial\phi/\partial n)U
  -\frac{1}{2}|\nabla\phi|^2 - g\,\eta_0(\alpha,t) + \tau\frac{\theta_\alpha}{s_{0,\alpha}} + P_{ext}.
  \end{align*}
Here $G(\eta)$ is the un-normalized Dirichlet-to-Neumann map, $\theta$ is the tangent angle of the surface, $s_\alpha$ is an arc-length parameter, $U,V$ are the normal and tangential derivatives at the surface and $P_{ext}$ is an external pressure term in which we can introduce damping or forcing on the equations.  We will consider especially a form of damping introduced by Clamond-Fructus-Frue-Kristiansen in \cite{clamond2005efficient}.  

 In the case of non-zero surface tension ($\tau > 0$) one can use a prescribed $P_{ext}$ to stabilize small waves similar to the work of Alazard et al, see \cite{alazard2018control}.  When $\tau = 0$, we can think of these conditions as numerical boundary conditions that absorb energy and allow for as little reflection as possible.   Generically, one numerically solves the water waves problem on a periodic domain of length $2 \pi$ and take $\omega \subset [0, 2\pi)$ a connected interval on which we will damp the fluid with corresponding indicator function $\chi_\omega$.  We will consider here the damping properties of the numerically effective damping term, 
 $$P_{ext} = \partial_x^{-1} (\chi_\omega \phi_x ),$$ 
 as proposed in \cite{clamond2005efficient}, which we will denote as Linear $H^{1/2}$ Damping of the water wave problem related to the order of regularity required to establish the model equation \eqref{eqn:dampwave} from a full paradifferential diagonalization of the water wave equations, see \cite{alazard2011water}.

Using the paradifferential formulation of the water waves developed in for instance \cite{alazard2011water}, one observes the following leading order linear model for damped gravity water waves
\begin{equation}
\label{eqn:dampwave}
\partial_t^2 u + |D| u + \chi \partial_t u = 0 .
\end{equation}
This model can be studied from the classical point of view of scattering theory and perfectly matched layers, though the non-local nature of the operator $|D|$ means that many known techniques fail and more refined tools are required.  To that end, we study \eqref{eqn:dampwave} here using propagation estimates in the study of semiclassical scattering operators, which have been developed quite thoroughly in the recent book \cite{dyatlov2019mathematical} for operators of the form $-\Delta + V$.  However, the non-locality of operators of the form $|D| + V$ results in some important modifications that we illuminate here.  Much of our analysis should be extendable to other non-local wave equation models with appropriate modifications.   The well-posedness of a nonlinear model related to \eqref{eqn:dampwave} in the setting of the water waves with surface-tension was established in the recent work of \cite{moon2022toy}, but the strength and speed of damping that arises from such a method is not clear.  Here, we are able to prove the polynomial decay of the energy for the linear model.  

{  As discussed in \cite{alazard2017stabilization}, there is a long-standing connection between damped wave equations, absorbing boundary conditions in numerical analysis and the notions of so-called control and observability estimates for a given equation on the support of the damping function.\footnote{See for instance \cite{alazard2017stabilization} or \cite{burq2019rough} for careful definitions of control/observability estimates if the reader is unfamiliar.}
While our approach here does not use such an estimates directly, some important surveys and results in this direction for a variety of models that have similar proof strategies include \cite{alabau2012wave,bardos1992sharp,burq2019rough,burq2004geometric,macia2021observability,phung2007polynomial,RaTa,zuazua2005propagation,zuazua2007controllability}.  We also highlight a $1d$ specific version of absorbing boundary condition was introduced in \cite{jennings2014water}, though we point out that the model we consider here can be easily generalized to higher dimensional water wave models.}

A related damping model is of the form
\begin{equation}
\label{eqn:dampwave_ham}
\partial_t^2 u + |D| u + |D|^{\frac12} (\chi |D|^{\frac12} \partial_t  u) = 0 .
\end{equation}
This results from a similar paradifferential diagonalization of a damping that is guaranteed to lead to nonlinear damping by consideration of the Hamiltonian energy for the water wave equations, see \cite{alazard2018stabilization}.  The techniques we apply here can likely be applied to study damping of this form with appropriate modifications, in particular with respect to the required regularity of the initial data.  However, for the sake of smoothness of exposition, we focus only on equation \eqref{eqn:dampwave} in our analysis below.

\subsection{Main results}
Here, we study a linear model for the damped water wave equation explicitly framed on a periodic domain, where we are able to give quantitative estimates on the damping rates of \eqref{eqn:dampwave}.  
To be precise, 
let $\mathbb T:=\RR/2\pi \mathbb Z$ be the 
circle and $\chi\in L^{\infty}(\mathbb T)$ 
satisfy $\chi\geq 0$. For $s>0$, we define the fractional Laplacian operator as follows

\begin{equation}\begin{gathered}
\label{eq: def-D}
    |D|^s\colon H^s(\mathbb T)\to L^2(\mathbb T), \ |D|^{s}u (x) := \sum_{n\in \mathbb Z} |n|^{s}\widehat{u}(n) e^{ i n x  }, \\
    \text{ for } u\in H^s(\mathbb T), \ \widehat{u}(n):=\frac{1}{2\pi}\int_0^{2\pi} e^{-i n x}u(x)dx.
\end{gathered}\end{equation}

For $(u_0,u_1)\in H^{\frac12}(\mathbb T)\times L^2(\mathbb T)$, we consider the damped fractional wave equation 
\begin{equation}
\label{eq: damp}
    (\partial_t^2+\chi\partial_t +|D|)u(t,x)=0, \ u(0,x)=u_0(x), \ \partial_tu(0,x)=u_1(x). 
\end{equation}
The energy of the solution to \eqref{eq: damp} is defined by 
\begin{equation}
\label{eq: def-energy}
    E(u,t):=\int_{\mathbb T} \left( ||D|^{\frac12}u(t,x)|^2+|\partial_t u(t,x)|^2 \right) dx.
\end{equation}

For a localized damping function $\chi$ (meaning that $\supp \chi\neq \mathbb T$), the geometric control condition fails. In this case, we show that for any $(u_0,u_1)\in H^{1}(\mathbb T)\times H^{\frac12}(\mathbb T)$, the energy of the solution decays as $1/t^2$. By constructing 
quasi-modes, we show that this polynomial rate is sharp for localized damping.

In order to explore the possibility of larger decay rates, we now consider damping with finitely many zeros. In this case, the energy decay rates depend on the ``switching on'' behavior of the damping function $\chi$ near zeros of $\chi$. To quantify this connection, we introduce the following definition:
\begin{defi}
\label{def: nondeg-chi}
We say $\chi \in C^{\infty}(\mathbb T)$, $\chi\geq 0$ has finite degeneracy, if $\chi$ has finitely many zeros $x_k$, $1\leq k\leq n$, and for each $x_k$, there exists $N_k>0$, such that
        \begin{equation}
            \chi^{(\ell)}(x_k)=0, \ 0\leq \ell\leq 2N_k-1, \ \chi^{(2N_k)}(x_k)>0.
        \end{equation}
\end{defi}

Given $\beta\in (0,1]$, 
we denote by $C^{0,\beta}(\mathbb T)$ the H\"older space of those continuous periodic 
functions $u\colon \mathbb{T}\to \xR$ such that 
$$
\lA u\rA_{C^{0,\beta}}=\sup_{x\in\mathbb{T}}\la u(x)\ra+
\sup_{x\neq y}\frac{\la u(x)-u(y)\ra}{\la x-y\ra}<+\infty.
$$
We are now ready to state the main results on the energy decay rates. 
\begin{theo}
\label{thm: decay}
Suppose $\chi\in C^{0,\beta}(\mathbb T)$ with $\beta>\frac12$, $\chi\geq 0$, $\chi\neq 0$. Then there exists $C>0$ such that for any $(u_0,u_1)\in H^{1}(\mathbb T)\times H^{\frac12}(\mathbb T)$, if $u$ solves \eqref{eq: damp}, then
\begin{equation}\label{i1}
    E(u,t)\leq \frac{C}{t^2} \left( \|u_0\|_{H^1}^2+\|u_1\|_{H^{\frac12}}^2 \right), \ t>0.
\end{equation}
Moreover, if $\chi\in C^{\infty}(\mathbb T)$ has finite degeneracy as in Definition \ref{def: nondeg-chi} and let $N$ be the maximal $N_k$ there, then for any $\gamma>0$, there exists $C>0$ such that for $(u_0, u_1)\in H^{1}(\mathbb T)\times H^{\frac12}(\mathbb T)$, we have 
    \begin{equation}
    \label{eq: full-spt-decay}
        E(u,t)\leq \frac{C}{t^{2+\frac{1}{N} -\gamma }}\left( \|u_0\|_{H^1}^2+\|u_1\|_{H^{\frac12}}^2 \right), \ t>0.
    \end{equation}
\end{theo}
\Remarks 
1. The $1/t^2$ energy decay rate for localized damping (see~\e{i1}) is sharp. 
This follows from the sharpness of the resolvent bound \eqref{eq: local-damp-bound} 
(see Remark after Theorem \ref{thm: resolvent-est}) and the equivalence between resolvent bounds and polynomial decay rates 
proved by Anantharaman--L\'eautaud (see \cite[Proposition 2.4]{anantharaman2014sharp}, which is also stated in \S \ref{ResToSG} in the current paper).  

\noindent
2.  Suppose $\chi\in C^{\frac{k}{2}}(\mathbb T)$ for $k \geq 2$. If $(u_0, u_1)\in H^{\frac{k+1}{2}} \times H^{\frac{k}{2}}$, then the polynomial rates in \eqref{i1}, \eqref{eq: full-spt-decay} can be improved to $\frac{C_k}{t^{2k}}$ and $\frac{C_k}{t^{(2+\frac{1}{N}) k-\gamma}}$, respectively. See Remark in \S \ref{ResToSG}.

\noindent
3. The first result in Theorem \ref{thm: decay} can be generalized to tori $\mathbb T^n$ of higher dimensions $n\geq 2$.
Indeed, suppose $\chi\in C^{\infty}(\mathbb T^n)$, $\chi\geq 0$, $\chi\neq 0$, and 
\begin{equation}
\label{eq: con-con}
    \text{ for any } (x,\xi)\in S^*\mathbb T^n, \text{ there exists } T>0, \text{ such that } x+T\xi \in \{ \chi>0 \},
\end{equation}
where $S^*\mathbb T^n \simeq \mathbb T^n\times \mathbb S^{n-1}$ is the cosphere bundle of $\mathbb T^n$.
Then there exists $C>0$ such that for any $(u_0, u_1)\in H^{1}(\mathbb T^n)\times H^{\frac12}(\mathbb T^n)$, we have 
\[ E(u,t)\leq \frac{C}{t^2} \left( \|u_0\|_{H^1}^2+\|u_1\|_{H^{\frac12}}^2 \right). \]
The proof is the same as in the 1D case presented in this paper. 

\noindent
4. For comparison, we recall the usual damped wave equation on $\mathbb T^n$
\begin{equation}
\label{eq: reg-wave}
(\partial_t^2 +\chi\partial_t -\Delta)u=0. 
\end{equation}
It is known that under the same dynamical condition \eqref{eq: con-con}, the energy of the solution to \eqref{eq: reg-wave} decays exponentially (see for instance \cite[Theorem 5.10]{zworski2012semiclassical}).  A result \cite{phung2007polynomial} for a damped wave equation on a bounded domain results in polynomial rates for damped wave equations without geometric control conditions provided an observability estimate holds and with very minimal regularity requirements on the damping function.

The energy decay rates for the damped equation \eqref{eq: damp} are closely related to the resolvent estimates (via semi-group theory or Fourier transform, see for instance \cite{anantharaman2014sharp, zworski2012semiclassical}) of the stationary operator 
\begin{equation} 
\label{eq: p-def}
P(\tau) := |D|-i\tau\chi-\tau^2
,\quad \ \tau\in \CC. 
\end{equation}
Theorem \ref{thm: decay} follows from the following resolvent bounds.
\begin{theo}
\label{thm: resolvent-est}
Suppose $\chi\in C^{0,\beta}(\mathbb T)$ with $\beta>\frac12$, $\chi\geq 0$, $\chi\neq 0$. Let $P(\tau)$ be as in \eqref{eq: p-def}. Then there exists $C>0$, such that
$\tau\in \RR$, $|\tau|>C$ implies
\begin{equation} 
\label{eq: local-damp-bound}
\|P(\tau)^{-1}\|_{L^2\to L^2}\leq C.
\end{equation}
Moreover, if $\chi\in C^{\infty}(\mathbb T)$ has finite degeneracy as in Definition \ref{def: nondeg-chi} and $N$ denotes 
the maximum of $N_k$ there, then for any $\gamma^{\prime}>0$, there exists $C>0$ such that $\tau\in \RR$, $|\tau|>C$ implies
\begin{equation}
\label{eq: finite-zero-bound}
   \|P(\tau)^{-1}\|_{L^2\to L^2}\leq C|\tau|^{-\frac{1}{2N+1}+\gamma^{\prime}}.
\end{equation}
\end{theo}

\Remark
Let us prove that~\e{eq: local-damp-bound} is optimal for cut-off function 
$\chi$ which does not have full support on $\mathbb T$ 
that is, $\chi$ is a localized damping function. To see this, 
let $a\in C^{\infty}(\mathbb T)$ be such that $a\neq 0$ and 
$\supp a\cap \supp \chi=\emptyset$. 
For all integers $k>0$, 
define $u_k(x):=a(x)e^{ikx}\in C^{\infty}(\mathbb T)$. Since 
$\supp a\cap \supp \chi=\emptyset$, we have 
\[ P(\sqrt{k})u_k(x) = [|D|,a]e^{ikx}. \]
We will establish in \S \ref{sec: pdo} that the commutator $[|D|,a]\in \Psi^0(\mathbb T)$. This means that $[|D|,a]: L^2(\mathbb T)\to L^2(\mathbb T)$ is a bounded operator and hence there is $C>0$ such that 
\[ \|P(\sqrt{k})u_k\|_{L^2} \leq C\|e^{ik\bullet}\|_{L^2}=(C/\|a\|_{L^2})\|u_k\|_{L^2}. \]
Therefore, there is no such $\delta>0$ and $C(\delta)>0$, such that 
$\tau\in \RR$, $|\tau|>C(\delta)$ implies
\[
\|P(\tau)^{-1}\|_{L^2\to L^2}\leq C(\delta) | \tau |^{-\delta}.
\]
This shows that \e{eq: local-damp-bound} is optimal for localized damping functions.

As we show in Lemma \ref{lem: mero}, 
\[ P(\tau)^{-1}: L^2(\mathbb T)\to H^1(\mathbb T), \ \tau\in \CC \]
is a meromorphic family of operators with finite rank poles. The poles of $P(\tau)^{-1}$ are called {\em resonances} for $P(\tau)$. We denote the set of resonances by $\mathscr R$. Using Theorem \ref{thm: resolvent-est} and Grushin problems, we give the following description of the distribution of resonances:
\begin{theo}
\label{thm: resonance}
Suppose $\chi\in C^{0,\beta}(\mathbb T)$ with $\beta>\frac12$, $\chi\geq 0$, $\chi\neq 0$. Then 
\begin{enumerate}[1.]
    \item There exists $C>0$, such that
    \begin{equation}
        \mathscr R\cap \{|\Re z|>C\}\subset \left\{ \tau\in \CC \ \left| \ \Im z>-C^{-1}|\Re \tau|^{-1} \right.\right\}.
    \end{equation}
    Moreover, if $\chi\in C^{\infty}(\mathbb T)$ has finite degeneracy as in Definition \ref{def: nondeg-chi}, then for any $\gamma^{\prime\prime}>0$, there exists $C=C(\gamma^{\prime\prime})>0$, such that 
    \begin{equation}
        \mathscr R\cap \{ |\Re z|>C \} \subset \{ \tau\in \CC \ | \ \Im z >-C^{-1}|\Re \tau|^{-\frac{2N}{2N+1}-\gamma^{\prime\prime}} \}.
    \end{equation}

    \item For $\nu>0$, we denote $P(\nu,\tau):=|D|-i\nu\tau \chi-\tau^2$ and $\mathscr R(\nu)$ the set of resonances of $P(\nu,\tau)^{-1}$. Then for each $k\in \mathbb Z$, $k>0$, there exists $\nu_k>0$,  an open neighborhood $U_k$ of $\sqrt{k}\in \CC$, and $\tau_{k,\pm}\in C^{\infty}((0,\nu_k);U_k)$, such that
    \begin{equation}
    \label{eq: res-exp}
        \mathscr R(\nu)\cap U_k= \{\tau_{k,+}(\nu), \ \tau_{k,-}(\nu)\}, \ \tau_{k,\pm}(\nu) = \sqrt{k} - \frac{\widehat{\chi}(0)\pm |\widehat{\chi}(2k)|}{2} i\nu+o(\nu).
    \end{equation}
\end{enumerate}
\end{theo}

\Remarks 
1. Despite the asymptotic expansion of the resonances in \eqref{eq: res-exp}, Theorem \ref{thm: resonance} does not imply the existence of a resonance-free strip with constant width for $P(\tau)^{-1}$, because the expansions are not uniform in $k$. 

\noindent
2. For the damped wave equation framed on a compact manifold, 
\[
\partial_t^2 u + \chi \partial_t u - \Delta_g u = 0, \]
where $\Delta_g$ is the usual Laplace-Beltrami operator, the distribution of resonances and corresponding energy decay rates have been studied in the works \cite{markus1982spectra}, \cite{sjostrand2000damp} and \cite{anantharaman2010spectral}. In \cite{markus1982spectra}, Markus--Matsaev established a Weyl law for the resonances in terms of counting how many resonances can exist at a given energy.  In \cite{sjostrand2000damp}, Sj\"ostrand further proved that the imaginary parts of the resonances ``concentrate'' (in a suitable sense) on the half average (with respect to the Liouville measure on the cosphere bundle of the manifold) of the damping function. In \cite{anantharaman2010spectral}, Anantharaman proved a fractal Weyl law for the resonances and studied several inverse problems. It is an important topic of future work to see if analogous results hold for damped fractional wave equations of the form studied here.

\subsection{Outline of Paper}

In \S \ref{sec:semiclassical}, we recall some properties of the semiclassical calculus for operators on the torus that we will require for our analysis.  The propagation estimates required to prove the resolvent estimate and the resulting resolvent mapping properties are proven in \S \ref{sec:resolvent}.  We prove a stronger resolvent estimate for damping functions that vanish to a given order at a finite number of points on $\mathbb{T}$ in \S \ref{sec:degdamp}.  In \S \ref{ResToSG}, we give an overview of the proof from \cite{anantharaman2014sharp} (simplified in our particular setting) that the resolvent bounds proved are equivalent to energy decay bounds for the damped fractional wave equation.  To give insight into the properties of the resolvent, in \S \ref{sec:resdist} we prove that the low energy resonances can be approximated well by a finite approximation that can be constructed explicitly using a Grushin problem.  Finally, in \S \ref{sec:numsims} we provide some numerical simulations demonstrating various aspects of our theorems in practice, both for the exact linear fractional wave model, as well as for water wave models with Clamond Damping.  This includes a means of approximating the low-energy resonances and comparing to the asymptotics in the previous section.  


\medskip

\noindent
{\bf Acknowledgements.}
We would like to thank Jared Wunsch for many helpful discussions, and Ruoyu P. T. Wang for showing us useful references. J.L.M was supported in part by NSF Applied Math Grant DMS-1909035 and NSF Applied Math Grant DMS-2307384.

\section{Semiclassical analysis on the circle}
\label{sec:semiclassical}

\subsection{Semiclassical pseudodifferential operators}
\label{sec: pdo}
We consider the following symbol class $S^k(T^*\mathbb T)$
\begin{equation*}
    S^k(T^*\mathbb T):=\left\{ a\in C^{\infty}(T^*\mathbb T) \ | \ |\partial_x^{\alpha}\partial_{\xi}^{\beta} a(x,\xi)|\leq C_{\alpha\beta} \langle\xi\rangle^{k-|\beta|}, \ C_{\alpha\beta}>0, \ \forall \alpha, \beta \right\},
\end{equation*}
here $T^*\mathbb T$ is the cotangent bundle of $\mathbb T$, $T^*\mathbb T\simeq \mathbb T_x\times \RR_{\xi}$, and $\langle \xi \rangle=\sqrt{1+\xi^2}$.
With the best constants $C_{\alpha\beta}$ as semi-norms, the class $S^k$ is a Fr\'echet space.
For $a\in S^k(T^*\mathbb T)$, which could depend on $h$ with semi-norms uniform in $h$, we define its semiclassical quantization by 
\begin{equation}
    \Op_h(a)u(x):=\int_{-\infty}^{\infty}\int_{-\infty}^{\infty} e^{\frac{i}{h}(x-y)\xi} a(h,x,\xi) u(y)dyd\xi.
\end{equation}
Let $\widehat{a}(k,\xi)$, $\widehat{u}(k)$ be the Fourier coefficients of $a(\bullet,\xi)$, $u$ as in \eqref{eq: def-D}. Then we have 
\begin{equation}
\label{eq: quant-fourier}
    \Op_h(a)u(x)=\sum_{n\in\mathbb Z}\sum_{\ell\in \mathbb Z} \widehat{a}(h,n-\ell,h\ell)\widehat{u}(\ell) e^{inx}.
\end{equation}

Similarly, we define its microlocal quantization by
\begin{equation*}\begin{split}
    \Op(a)u(x)
     :=\int_{-\infty}^{\infty}\int_{-\infty}^{\infty} e^{i(x-y)\xi} a(h,x,\xi) u(y)dyd\xi 
    = \sum_{n\in \mathbb Z}\sum_{\ell\in Z} \widehat{a}(h,n-\ell, \ell)\widehat{u}(\ell) e^{inx}.
\end{split}\end{equation*}
We also use the notations $a(x,D)$, $a(x,hD)$ for the microlocal or semiclassical quantization of $a$.
Let $\Psi_h^k(\mathbb T)$ be the set of pseudodifferential operators that  consists of semiclassical quantizations of all symbols in $S^k(T^*\mathbb T)$.
We define the semiclassical symbol map 
\[ \sigma_h: \Psi_h^k(\mathbb T) \to S^k(T^*\mathbb T)/S^{k-1}(T^*\mathbb T), \ A \mapsto [a]. \]
Later we will identify $[a]$ with $a$ if no ambiguity. 

\Remark 
$|D|$ is a microlocal operator with symbol $a(\xi)\in S^1(T^*\mathbb T)$ such that $a(k)=|k|$, $k\in \mathbb Z$. However, $h|D|$ is {\em not} a semiclassical operator, as its ``semiclassical symbol'' $|\xi|$ is not a smooth function.

Consider a family of symbols $a=\{a(h,x,\xi)\, | \, 0<h<h_0\}$ which is bounded in $S^k$ and is $C^{\infty}$ in $h$. Let $A$ be the semiclassical quantization of $a$. The semiclassical wavefront set $\WF_h(A)$ of $A$ is defined to be the essential support of $a$, that is, $(x_0,\xi_0)\notin \WF_h(A)$ if and only if there exists a neighborhood $U$ of $(x_0,\xi_0)$ in $T^*\RR$, such that for any $\alpha, \beta\in \NN$,
\[ \partial_x^{\alpha}\partial_{\xi}^{\beta}a(h,x,\xi)=O(h^{N}\langle 
\xi \rangle^{-N}), \text{ for all } (x,\xi)\in U, \ N\in \NN. \]
We also define the semiclassical elliptic set of $A$ by
\[ \Ell_h(A):=\{ (x,\xi) \ | \ \langle \xi \rangle^{-k}\sigma_h(A)(x,\xi)\neq 0 \}. \]

We record the following formula for the symbol calculus of operator compositions: if $a\in S^k(T^*\mathbb T)$, $b\in S^{\ell}(T^*\mathbb T)$, then 
\[\begin{gathered} 
\Op_h(a)\Op_h(b)=\Op_h(a\# b), \ a\# b \in S^{k+\ell}(T^*\mathbb T), \\
a\# b = \sum_{k=0}^{N} \frac{(-ih)^{j}}{j!}\partial_{\xi}^j a(x,\xi)\partial_x^j b(x,\xi) + O_{S^{k+\ell-j}(T^*\mathbb T)}(h^{N+1}), \ \text{for all} \ N\in \NN.
\end{gathered}\]
In particular, we can compute the commutator of two pseudodifferential operators 
\begin{equation}\label{commutator}
[\Op_h(a), \Op_h(b)] \in h\Psi^{k+\ell-1}_h(\mathbb T), \ \sigma_h(h^{-1}[\Op_h(a), \Op_h(b)]) = -i \{a,b\}  
\end{equation}
where $\{a,b\} := \partial_{\xi}a\partial_x b-\partial_x a\partial_{\xi}b$.

\subsection{Semiclassical Fourier multipliers}
When the symbol $a$ in \eqref{eq: quant-fourier} does not depend on $x$, we say $\Op_h(a)$ is a semiclassical Fourier multiplier. In this section, we generalize the definition of semiclassical Fourier multipliers to bounded symbols and prove a commutator estimate for semiclassical Fourier multipliers and functions with H\"older continuity.

 Given a bounded function 
$a\colon \xR\to\xR$, we define the semiclassical 
Fourier multiplier $a(hD)$ by 
$$
\widehat{a(hD)u}(k)=a(hk)\widehat{u}(k)
$$
for $u\in L^2(\mathbb T)$. We call $a$ the symbol of the semi-classical Fourier multiplier $a(hD)$.

\begin{prop}\label{P:comm}
Let $a \in C^\infty_b(\mathbb{R})$ 
be a smooth function, bounded together with all its derivatives, which in addition vanishes on a neighborhood of the origin. Consider two real numbers $0<\alpha <\beta \leq 1$. 
There exists a constant $C>0$ such that, for all 
$f\in C^{0,\beta}(\xT)$ and for all 
$u\in L^2(\mathbb{T})$, 
\begin{equation}\label{eq:comm}
\lA f a(hD) u-a(hD)(fu)\rA_{L^2}\le C h^{\alpha} \lA f\rA_{C^{0,\beta}}\lA u\rA_{L^2}.
\end{equation}
\end{prop}
\begin{proof}
We introduce the symbols $a_h$ defined by $a_h(\xi)=a(h\xi)$. 
The key point is to observe  
that~$\left\{\,  h ^{-\alpha} a_{h} \mid 0< h \le 1 \,\right\}$ is a bounded family in~$S^{\alpha}$. The wanted estimate~\eqref{eq:comm} will then be a direct consequence of the following lemma.

\begin{lemm}\label{commutateur_S^1_f}
Consider two real numbers $m$ and $\alpha$ such that~$0<\alpha<\beta\le 1$. 
For any bounded subset~$\mathcal{B}$ of~$S^{\alpha}$, 
there exists a constant~$K$ 
such that for all symbol~$q\in\mathcal{B}$, 
all $f\in C^{0,\beta}(\mathbb{T})$, 
and all $u\in L^{2}(\mathbb{T})$,
\begin{equation}\label{commutator:estimateusual}
\left\lVert\mathcal{Q}(fu)-f\mathcal{Q} u\right\rVert_{L^{2}}
\le K\left\lVert f\right\rVert_{C^{0,\beta}(\mathbb{T})}
\left\lVert u\right\rVert_{L^{2}},
\end{equation}
where $\mathcal{Q}$ is the Fourier multiplier with symbol~$q$. 
\end{lemm}
\begin{proof}To prove this result, it 
is convenient to use the paradifferential calculus of Bony~\cite{Bony} and the Littlewood-Paley decomposition. We start by introducing some notations. 
Fix a function 
$\Phi\in C^\infty_0(\xR)$ with support in the interval $[-1,1]$ and 
equal to $1$ when $|\xi|\le 1/2$. Then set 
$\phi(\xi)=\Phi(\xi/2)-\Phi(\xi)$ which is supported in $\{\xi\in\xR \ | \ 1/2\le |\xi|\le 2\}$. 
Then, for all $\xi\in\xR$, one has $\Phi(\xi)+\sum_{j\in \mathbb{N}}\phi(2^{-j}\xi)=1$, 
which one can use to decompose tempered distribution (this setting includes in particular periodic function 
$u\in L^2(\xT)$). 
For $u\in \mathcal{S}'(\xR)$, 
we set
$$
\Delta_{-1}u
=\mathcal{F}^{-1}(\Phi(\xi)\widehat{u}),\qquad 
\Delta_j u=\mathcal{F}^{-1}(\phi(2^{-j}\xi)\widehat{u})\quad\text{for}\quad j\in\mathbb{N}.
$$ 
We also use the notation $S_j u =\sum_{-1\le p\le j-1} \Delta_pu$ for $j\ge 0$ (so that $S_0u=\Delta_{-1}u=\Phi(D)u$). 

Given a function $f$, denote by $f^{\flat}$
denotes the multiplication operator $ u\mapsto fu$ and denote by~$T_f$ the operator of paramultiplication by~$f$, defined by 
$$
T_f u=\sum_{j\ge 1}S_{j-1}(f)\Delta_j u.
$$
Now rewrite the commutator~$\left[ \mathcal{Q}, f^{\flat} \right]$ as
\begin{equation*}
\left[\mathcal{Q},T_{f}\right]+\mathcal{Q}(f^{\flat}-T_{f})-(f^{\flat}-T_{f})\mathcal{Q}.
\end{equation*}
The claim then follows from the bounds 
\begin{align}
& \forall \sigma\in\mathbb{R}, \quad \quad \quad
\left\lVert\left[\mathcal{Q},T_f \right]\right\rVert_{H^{\sigma} \rightarrow H^{\sigma-\alpha +\beta}} \le 
c_{1}(q,\sigma)\left\lVert f\right\rVert_{C^{0,\beta}}, \label{Com_Paradif_1} \\
&  \forall \sigma\in [-\alpha,0],\quad \big\lVert f^{\flat}-T_{f}\big\rVert_{H^{\sigma}\rightarrow H^{\sigma+\alpha}} \le c_2(\sigma)\lA f\rA_{C^{0,\beta}}, \label{Com_Paradif_2} \\
& \forall\sigma\in\mathbb{R}, \quad \quad \quad 
\left\lVert\mathcal{Q}\right\rVert_{H^{\sigma} \rightarrow H^{\sigma-\alpha}} \lesssim\sup\nolimits_{\xi} \la \langle \xi\rangle^{-\alpha}q(\xi)\ra,
\label{Com_Paradif_3}
\end{align}
where~$c_{1}(\cdot,\sigma)\colon S^{\alpha}\rightarrow \mathbb{R}_+$ maps bounded
sets to bounded sets. The estimate~\eqref{Com_Paradif_3} is a direct result of Plancherel Theorem.
We refer the reader to~\cite[Proposition 10.2.2]{HorL} and 
\cite[Theorem 9.6.4$'$]{HorL} for the proof of~\eqref{Com_Paradif_1}. Let us now prove~\eqref{Com_Paradif_2}. To do so, observe that
$$
fg-T_{f}g=\sum_{j,p\ge -1}(\Delta_j f)(\Delta_{p}g)-\sum_{-1\le j\le p-2}(\Delta_j f)(\Delta_{p}g)
=\sum_{j\ge -1} (S_{j+2}g)\Delta_{j}f.
$$
We then use the Bernstein's inequality and the characterization of Sobolev and H\"older spaces in 
terms of Littlewood-Paley decomposition, to write
\begin{alignat*}{2}
\lA fg-T_{f}g\rA_{H^{\sigma+\alpha}}
&\le \sum \lA (S_{j+2}g)(\Delta_{j}f)\rA_{H^{\sigma+\alpha}} &&\\
&\les \sum 2^{j(\sigma+\alpha)}\lA (S_{j+2}g)(\Delta_{j}f)\rA_{L^2} \qquad && (\text{since }\sigma+\alpha\ge 0)\\
&\le \sum 2^{j(\sigma+\alpha)} \lA S_{j+2}g\rA_{L^2}\lA \Delta_{j}f\rA_{L^\infty}&&\\
&\les \sum 2^{j\alpha}\lA g\rA_{H^{\sigma}}2^{-j\alpha}\lA f\rA_{C^{0,\beta}}&&(\text{since }\sigma\le 0)\\
&\les \lA g\rA_{H^{\sigma}}\lA f\rA_{C^{0,\beta}},
\end{alignat*}
where we used the assumption $\alpha<\beta$ to insure that the series $\sum 2^{j(\alpha-\beta)}$ converges.
\end{proof}
This concludes the proof of Proposition~\ref{P:comm}.
\end{proof}

\section{Resolvent estimates for localized damping}
\label{sec:resolvent}

This section is devoted to proving the resolvent bound for localized damping, that is, the first part of Theorem \ref{thm: resolvent-est}. This resolvent bound gives $1/t^2$ energy decay for solutions to \eqref{eq: damp} when $(u_0, u_1)\in H^{1}\times H^{\frac12}$ using \cite{anantharaman2014sharp}, and the proof of the energy decay is streamlined in \S \ref{ResToSG}.

To take advantage of semiclassical analysis, we introduce the semiclassical rescaling 
\begin{equation}
\label{eq: semi-scale}
\tau = \frac{z}{\sqrt{h}}, \ h>0, \ z\in \CC, 
\end{equation}
and define 
\begin{equation}\label{eq: phz}
    \mathcal P(h,z) := h|D| -i\sqrt{h} z\chi - z^2.
\end{equation}
We notice that 
\[ P(\tau) = h^{-1}\mathcal P(h,z). \]
We start by stating an equivalent version of the first part of Theorem \ref{thm: resolvent-est} in the semiclassical scale.
\begin{prop}
\label{prop: semi-loc}
Suppose $ \chi\in C^{0,\beta}(\mathbb T) $ with $\beta>\frac12$. Let ${\mathcal P}(h,z)$ be as in \eqref{eq: phz}. Then there exist 
$h_0>0$, $\delta>0$ and $C>0$ such that, for all $0<h<h_0$, all complex number $z\in S_\delta=\{x+iy \ | \ x\in (1-\delta, 1+\delta), y \in (-\delta h,\delta h)\}$ 
and for all 
$u\in C^{\infty}(\mathbb{T})$, there holds
\begin{equation}\label{n20}
\lA u\rA_{L^2}\leq Ch^{-1}\lA \mathcal{P}(h,z)u\rA_{L^2}.
\end{equation}
\end{prop}
\begin{proof}[Proof of Proposition \ref{prop: semi-loc}]
We denote by $C$ various constants independent of $h$ and whose value may change from line to line. 
We write $A\les B$ to say that $A\le CB$ for such a constant $C$.

We first assume that $z$ is a real number with $z\in (1-\delta, 1+\delta)$ for some $\delta>0$  sufficiently small.

{\bf 1. Estimate of $\chi u$.}

We claim that
\be\label{n0}
\lA \chi u \rA_{L^2}^2
\les h^{-1}\|{\mathcal P}(h,z)u\|_{L^2}\lA u\rA_{L^2}.
\ee
To see this, observe that, by definition of 
$\mathcal{P}(h,z)$ we have
$$
\chi u =-i z^{-1} h^{-\frac12}\left(h\la D\ra u -\mathcal{P}(h,z)u-z^2 u\right).
$$
Consequently, by taking the $L^2(\mathbb{T})$-scalar product with $u$, we get
\[
\langle \chi u,u \rangle=\RE \langle \chi u,u \rangle = 
-z^{-1}h^{-\frac12}  \IM\langle {\mathcal P}(h,z)u,u \rangle.
\]
Therefore, it follows from the Cauchy-Schwarz inequality that
$$
0\le \langle \chi u,u \rangle    \leq 
\frac{1}{1-\delta}h^{-\frac12}\lA{\mathcal P}(h,z)u\rA_{L^2}\lA u\rA_{L^2}.
$$
Since $\lA \chi u\rA_{L^2}^2\le \lA \chi\rA_{L^\infty}\langle \chi u,u \rangle$, this immediately implies the wanted estimate~\e{n0}.

\bigbreak

{\bf 2. Propagation estimates.}

\smallbreak

The estimate of the remaining component $(1-\chi)u$ 
is divided into two steps. 
We begin with the most delicate part, which consists in estimate the microlocal component of $(1-\chi)u$ where the operator $\mathcal{P}(h,z)$ is not elliptic. The analysis will therefore rely on a propagation argument. 

More precisely, consider a cut-off function 
$g\in C^{\infty}(\mathbb{R}; [0,1])$ satisfying
\[ g(\xi)
=
\left\{
\begin{aligned}
&0 \quad&&\text{for } \ \ |\xi|\leq \tfrac15 \text{ or } |\xi|\geq 5,\\ 
& 1\quad &&\text{for } \tfrac14\leq |\xi|\leq 4.
\end{aligned}
\right.
\]
We want to estimate the $L^2$-norm of 
$g(hD)((1-\chi)u)$. We claim that, for some exponent $\nu>0$, 
\be\label{N1}
\lA g(hD)\big((1-\chi)u\big)\rA_{L^2}^2
\les 
\lA \chi  u\rA_{L^2}^2 + 
h^{-1}\|{\mathcal P}(h,z)u\|_{L^2}\lA u\rA_{L^2} +  h^\nu\lA u\rA_{L^2}^2.
\ee
To prove this claim we use two different kinds of 
localization. 

{\em Localization in frequency.} We further 
decompose the problem into waves traveling to the left and waves traveling to the right. To do so, consider two cut-off functions 
$g_{\pm}\in C_c^{\infty}(\mathbb{R}; [0,1])$ 
with $g_{-}(\xi)=g_+(-\xi)$ and such that
\[ g_+(\xi)
=
\left\{
\begin{aligned}
&0 \quad&&\text{for } \ \xi\in (-\infty,\tfrac15)\cup (5,\infty),\\ 
& 1\quad &&\text{for } \xi\in [\tfrac14,4].
\end{aligned}
\right.
\]
To obtain the wanted estimate~\e{N1}, 
it is sufficient to prove that 
$ \lA g_+(hD)\big((1-\chi)u\big)\rA_{L^2}$ and 
$\lA g_-(hD)\big((1-\chi)u\big)\rA_{L^2}$ are bounded by the right-hand side of \e{N1}. 

These two terms will be treated similarly (see the explanations at the end of this step) and for notational simplicity we focus on the estimate of $g_+(hD)((1-\chi)u)$. 
Our aim is thus to prove that, for some exponent $\nu>0$, we have
\be\label{n10}
\lA g_+(hD)\big((1-\chi)u\big)\rA_{L^2}^2
\les 
\lA \chi u\rA_{L^2}^2 + 
h^{-1}\lA{\mathcal P}(h,z)u\rA_{L^2}\lA u\rA_{L^2} + h^\nu\lA u\rA_{L^2}^2.
\ee

{\em Localization in space.} 
In addition to the previous localization in frequency, we see 
that to prove~\e{n10}, 
by using a suitable partition of unity, it is sufficient to prove that, for any point $x_0\in \supp(1-\chi)$, there exist $\nu>0$ and a 
function $b\in C^\infty(\mathbb{T};[0,1])$ with $b(x_0)>0$ such that
\be\label{n2}
\lA g_+(hD)(b u)\rA_{L^2}\les\lA \chi u\rA_{L^2}^2 + 
h^{-1}\lA{\mathcal P}(h,z)u\rA_{L^2}\lA u\rA_{L^2} + h^\nu\lA u\rA_{L^2}^2.
\ee

We will use suitable cut-off functions $b$, 
as given by the following 

\begin{lemm}
\label{lem: cut-off}
Assume that $\chi(x_0)=0$. Then there exist two functions $\varphi, b \in C^{\infty}(\mathbb T)$ such that
$$
1\le \varphi \le 3,\  0\le b\le 1, \  b(x_0)>0,
$$
and moreover $\varphi$ is such that its derivative 
satisfies $\varphi^{\prime}=\psi_+ - \psi_-$ for some 
functions $\psi_{\pm}\in C^{\infty}(\mathbb T)$ satisfying 
$\psi_{\pm}\geq 0$ and
\begin{equation*}
\supp \psi_+\cap \supp \psi_-=\emptyset,\ 
\ \chi|_{\supp \psi_+}>0, \ \psi_-|_{\supp b}>0.
\end{equation*}
\end{lemm}

\begin{figure}[t]
    \includegraphics[scale=0.6]{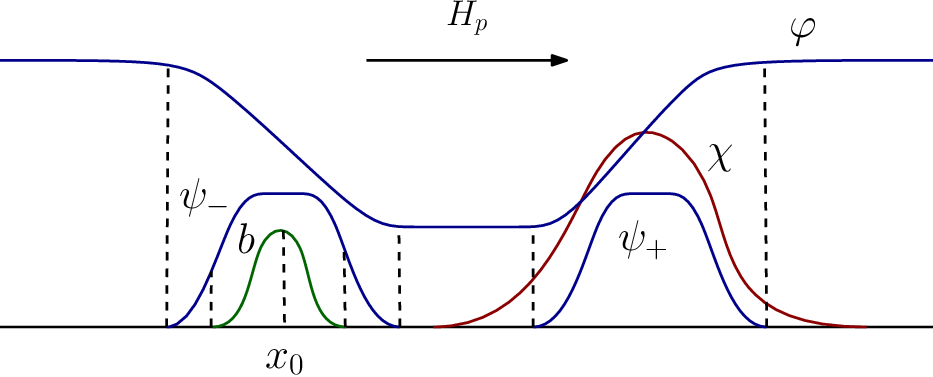}
    \caption{Auxiliary functions $\varphi$, $b$ constructed in Lemma \ref{lem: cut-off} with frequencies localized near $\xi=1$. The arrow indicates the direction of the Hamiltonian flow for $p=|\xi|$ near $\xi=1$.}
    \label{fig: cut-off}
\end{figure}

\begin{proof}[Proof of Lemma \ref{lem: cut-off}]
Introduce a $2\pi$-periodic function $\kappa\in C^\infty(\mathbb{T};[0,1])$ which is even and such that
\[
\kappa(x)=
\left\{
\begin{aligned}
&1 \quad&&\text{for } \ \la x\ra\le \tfrac14,\\ 
& 0\quad &&\text{for } \tfrac12\le \la x\ra\le \pi.
\end{aligned}
\right.
\]
Given three parameters $\alpha,\alpha'$ and $r$ to be determined, define
$$
\psi_-(x)=\alpha \kappa\left(\frac{x-x_0}{2r}\right),\quad b(x)=\alpha'\kappa\left(\frac{x-x_0}{r}\right).
$$
Now pick $x_1$ such that $\chi(x)>0$ for all $x\in [x_1-2r,x_1+2r]$  and set
$$
\psi_+(x)=\alpha \kappa\left(\frac{x-x_1}{2r}\right).
$$
Now let $\theta$ be the unique function 
$\theta\in C^\infty(\mathbb{T})$ with mean value $0$ and such that $\theta'=\psi_+-\psi_-$. We then set $\varphi(x)=2+\theta(x)$ and choose $\alpha,\alpha'$ and $r$ small enough.
\end{proof}

{\em Commutator argument.} 
Given the functions $\varphi=\varphi(x)$ 
and $g_+=g_+(\xi)$ 
as introduced above, we consider 
the operator 
$G_{+}$ defined by
\be\label{n60}
G_{+}u=\varphi g_+(hD)u.
\ee
The idea is to exploit the fact that $G_+^*G_+$ is self-adjoint to 
write 
$\Im\langle {\mathcal P}(h,z)u, G_+^*G_+u \rangle $ under the form of a commutator:
\begin{align}
     \Im\langle {\mathcal P}(h,z)u, G_+^*G_+u \rangle 
    &= \frac{1}{2i} \Big( 
    \langle {\mathcal P}(h,z)u, G_+^*G_+u \rangle
    -\langle {G_+^*G_+u,\mathcal P}(h,z)u \rangle \Big)\notag\\
&=\frac{1}{2i} \Big( \langle G_+^*G_+\mathcal{P}(h,z)u, u \rangle
    -\langle {\mathcal P}(h,z)^*G_+^*G_+u,u \rangle \Big)\notag\\
&=    \big\langle -\tfrac{1}{2i}[h|D|, G_+^*G_+]u,u \big\rangle
     -z\sqrt{h}  \Re\langle \chi u, G_+^*G_+u \rangle.\label{eq: com}
\end{align}

We then 
notice that by the assumption on the support of $g_+$, we have 
$\la h D\ra g_+(hD)  = h D g_+(hD)$ which in turn implies that
$$
\la  D\ra G_+^*= DG_+^*\quad\text{and}\quad
G_+\la D\ra= G_+ D \quad \text{with }D=\frac{1}{i}\partial_x.
$$
We thus end up with
\begin{equation*}
    -\frac{1}{2i}[h|D|, G_+^*G_+]=-\frac{1}{2i}[hD, G_+^*G_+].
\end{equation*}
Once this formula is established, 
one can compute this commutator using only the Leibniz formula. Indeed, directly from the definitions of $D=\partial_x/i$ and $G_+$ (see \e{n60}), 
we have (noticing that $g_+(hD)^*=g_+(hD)$)
\begin{align*}
-\frac{1}{2i}[hD, G_+^*G_+]u&=-\frac{h}{2i}\frac{1}{i}\big[\partial_x, g_+(hD) \big(\varphi^2 g_+(hD)\cdot\big)\big]u\\
&=h g_+(hD)\big(\varphi \varphi' g_+(hD)u\big).
\end{align*}
By combining the previous identities and using again the fact that $g_+(hD)^*=g_+(hD)$, we conclude that
$$
\big\langle -\tfrac{1}{2i}[h|D|, G_+^*G_+]u,u \big\rangle=h 
\big\langle \varphi\varphi' g_+(hD)u,g_+(hD)u\big\rangle.
$$
Now we use the special form of the function $\varphi$, that is the fact that 
$\varphi'=\psi_+-\psi_-$. Introduce 
the semiclassical operators 
$E_+$ and $F_+$ defined by 
$$
E_+u:=\varphi \psi_+ g_+(hD)u\quad,\quad
F_+u:=\varphi \psi_- g_+(hD)u.
$$
Then, we have
\begin{equation}
\label{eq: prop-com}
    \langle -\tfrac{1}{2i}[h|D|, G_+^*G_+]u,u \rangle = h\langle E_+u,g_+(hD)u \rangle-h\langle F_+u,g_+(hD)u \rangle.
\end{equation}

Having analyzed the first term in the right-side of \e{eq: com}, we now estimate the second one. To do so, we begin by writing the latter under the form
\begin{equation}
\Re\langle \chi u, G_+^*G_+ u \rangle = \langle \chi G_+ u, G_+ u \rangle + \langle \Re(G_+^*[G_+,\chi])u, u \rangle.
\end{equation}
Since $\chi\ge 0$, one has the obvious inequality
$$
\langle \chi G_+ u, G_+ u \rangle\ge 0.
$$
To estimate the commutator 
$[G_+,\chi] u$ we shall make use of the assumption that $\chi$ belongs to the H\"older space 
$C^{0,\beta}(\mathbb{T})$ for some exponent $\beta\in (1/2,1]$. 
Let $0<\eps<\beta-1/2$. 
It follows from Proposition~\ref{P:comm} applied with $\alpha=1/2+\eps$ that 
\begin{equation}
\lA [G_+,\chi] u\rA_{L^2}\leq C h^{\frac12 +\eps} \|u\|_{L^2}.
\end{equation}
Since $\lA G_+\rA_{L^2\to L^2}$ is bounded uniformly in $h$, it follows that 
\begin{equation}
\label{eq: prop-lot}
    z\sqrt{h}  \Re\langle \chi u, G_+^*G_+u \rangle
    \geq - Ch^{1+\eps}\lA u\rA_{L^2}^2.
\end{equation}


Now, by combining \eqref{eq: com} together with 
\eqref{eq: prop-com} and \eqref{eq: prop-lot}, 
we find 
\begin{equation*}
    h\langle F_+ u,g_+(hD)u \rangle\leq h\langle E_+ u, g_+(hD)u \rangle + |\langle
{\mathcal P}(h,z)u, G_+ ^*G_+u \rangle|+O(h^{1+\eps})\lA u\rA^2_{L^2}.
\end{equation*}
By using again the fact that 
$\lA G_+\rA_{L^2\to L^2}$ is bounded uniformly in $h$, by dividing each side of the previous inequality by $h$ and using the Cauchy-Schwarz inequality, this yields
\begin{equation*}
    \langle F_+ u, g_+(hD)u \rangle\leq \langle E_+ u,g_+(hD) u \rangle + h^{-1}\lA {\mathcal P}(h,z)u\rA_{L^2}\lA u \rA_{L^2}+O(h^\eps)\lA u\rA^2_{L^2}.
\end{equation*}

It remains to bound $\langle F_+ u, u \rangle$ (resp.\ $\langle E_+ u, u \rangle$) from below (resp.\ above). 
To do so, set $f=\sqrt{\psi_-}$. Since $f$ belongs to $C^{0,1/2}(\mathbb{T})$, it follows from Proposition~\ref{P:comm} applied with $\beta=1/2$ and $\alpha$ any arbitrary real number in 
$(0,\beta)$, that
$$
\langle F_+ u, g_+(hD)u \rangle=
\langle \varphi g_+(hD)(fu), g_+(hD)(fu) \rangle
+O(h^\alpha)\lA u\rA_{L^2}^2.
$$
Since $\varphi\ge 1$, we deduce that
$$
\langle F_+ u, g_+(hD)u \rangle\ge 
\lA  g_+(hD)(\sqrt{\psi_-} u) \rA_{L^2}^2
+O(h^\alpha)\lA u\rA_{L^2}^2.
$$
Now consider the function $b$ as given by Lemma~\ref{lem: cut-off}. Since $b$ can be written under the form 
$$
b=f\sqrt{\psi_-} \quad\text{with}\quad f=\frac{b}{\sqrt{\psi_-}}\in C^\infty(\mathbb{T}),
$$
we have
$$
g_+(hD)(b u)=f g_+(hD)(\sqrt{\psi_-} u) +\big[ 
g_+(hD),f\big] u.
$$
Using again Proposition~\ref{P:comm} to estimate the commutator, we get that 
$$
\langle F_+ u, g_+(hD)u \rangle\ge 
\lA  g_+(hD)(b u) \rA_{L^2}^2
+O(h^\alpha)\lA u\rA_{L^2}^2.
$$
On the other hand, by applying  
Proposition~\ref{P:comm} with $f=\sqrt{\varphi \psi_+}\in C^{0,1/2}(\mathbb{T})$, we deduce that
$$
\langle E_+ u, g_+(hD)u \rangle=
\lA g_+(hD)(\sqrt{\varphi \psi_+}u)\rA_{L^2}^2
+O(h^\alpha)\lA u\rA_{L^2}^2.
$$
This immediately implies that
$$
\langle E_+ u, g_+(hD)u \rangle\le 
\lA \sqrt{\varphi \psi_+}u\rA_{L^2}^2
+O(h^\alpha)\lA u\rA_{L^2}^2\le C 
\lA \chi u\rA_{L^2}^2 +Ch^\alpha\lA u\rA_{L^2}^2.
$$
Consequently, we end up with 
\begin{equation}\label{3.18}
\lA g_+(hD)(b u)\rA_{L^2}\les
\lA \chi u\rA_{L^2}^2 + 
h^{-1}\|{\mathcal P}(h,z)u\|_{L^2}\lA u\rA_{L^2} + (h^\alpha+h^\eps)\lA u\rA_{L^2}^2.
\end{equation}

This completes the proof of \e{n2} which in turn completes the proof of \e{n1}.

Lastly, to prove similar estimates for $\|g_-(hD)((1-\chi)u)\|_{L^2}$, it suffices to construct auxiliary functions $\varphi$, $b$ as in Lemma \ref{lem: cut-off} with mere changes (see Figure \ref{fig: cut-off-m})
\[ \varphi^{\prime}=\psi_- - \psi_+, \  \chi|_{\supp \psi_-}>0, \ \psi_+|_{\supp b}>0. \]

\begin{figure}[t]
    \includegraphics[scale=0.6]{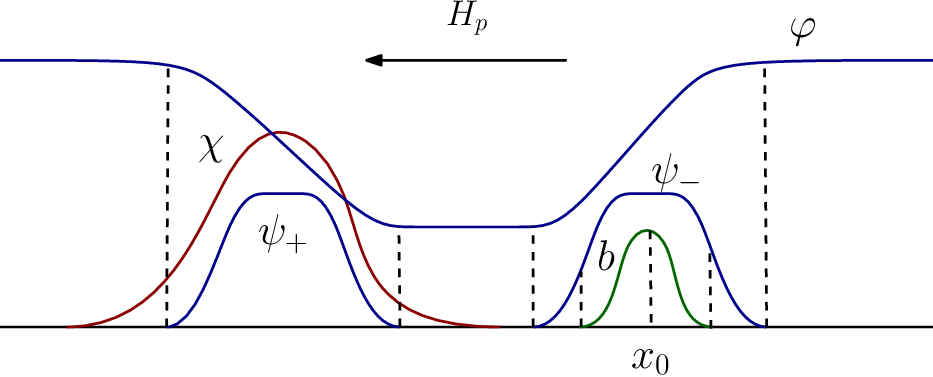}
    \caption{Auxiliary functions $\varphi$, $b$ used in the proof of the propagation estimates with frequencies localized near $\xi=-1$. The arrow indicates the direction of the Hamiltonian flow for $p=|\xi|$ near $\xi=-1$.}
    \label{fig: cut-off-m}
\end{figure}

{\bf 3. Elliptic estimates.}

Having estimated the main contribution of the frequencies of size near $1$ in the semiclassical scale, we now turn to the estimation of the low and high frequency components. More precisely, we want to estimate the $L^2$-norms of 
$G_0u$ and $G_\infty u$ where $G_0=g_0(hD)$ and $G_\infty=g_\infty(hD)$ are semiclassical Fourier multipliers with symbols $g_0=g_0(\xi)$ and 
$g_\infty=g_\infty(\xi)$ in $C^{\infty}(\mathbb{R};[0,1])$, such that
\begin{equation}\label{eq: g0i}
    \begin{gathered}
        g_{0}(\xi)=1 \ \text{for}\ |\xi|\leq 1/5, \ g_{0}(\xi)=0 \ \text{for}\ |\xi|\geq 1/4;\\
        g_{\infty}(\xi)=0 \ \text{for} \ |\xi|\leq 4, \ g_{\infty}(\xi)=1 \ \text{for}\ |\xi|\geq 5.
    \end{gathered}
\end{equation}

{\em High frequency estimate.} 
We begin by estimating $G_{\infty}u$. To do so, we write that, by the definition of $\mathcal{P}(h,z)$,
$$
G_\infty \la hD\ra u=G_\infty \mathcal{P}(h,z)u+ z^2 G_\infty u + i z G_\infty \sqrt{h}\chi u.
$$
Since $\lA G_\infty\rA_{L^2\to L^2}\le 1$, we have
$$
\lA G_\infty \la hD\ra u\rA_{L^2}\le \lA 
\mathcal{P}(h,z)u\rA_{L^2}+(1+\delta)^2\lA G_\infty u\rA_{L^2}
+(1+\delta)\sqrt{h}\lA \chi\rA_{L^\infty}\lA u\rA_{L^2}.
$$
On the other hand, since $\la \xi\ra\ge 4$ on the support of $g_\infty(\xi)$, 
directly from the Plancherel theorem, we have
$$
4\lA G_\infty u\rA_{L^2}\le \lA G_\infty \la hD\ra u\rA_{L^2}.
$$
Therefore, for $h$ and $\delta$ small enough, we get
$$
\lA G_\infty u\rA_{L^2}^2\le \lA 
\mathcal{P}(h,z)u\rA_{L^2}^2+\tfrac{1}{3} \lA u\rA_{L^2}^2.
$$

{\em Low frequency estimates.}
The low frequency component $G_0u$ is estimated in a similar way. 
Namely, we write
$$
z^2 G_0 u =G_0 \la hD\ra u - G_0 \mathcal{P}(h,z)u-i z G_0\sqrt{h}\chi u,
$$
and then observe that 
$\lA G_0\rA_{L^2\to L^2}\le 1$, to obtain
$$
(1-\delta)^2\lA G_0 u\rA_{L^2}\le \lA G_0 \la hD\ra u\rA_{L^2}+\lA 
\mathcal{P}(h,z)u\rA_{L^2}+
(1+\delta)\sqrt{h}\lA \chi\rA_{L^\infty}\lA u\rA_{L^2}.
$$
Since $\la \xi\ra\le 1/4$ on the support of $g_0(\xi)$, 
directly from the Plancherel theorem, we have
$$
\lA G_0 \la hD\ra u\rA_{L^2}\le \tfrac{1}{4} \lA G_0u\rA_{L^2}.
$$
Therefore, for $h$, $\delta$ small enough, we get
$$
\lA G_0 u\rA_{L^2}^2\le \lA 
\mathcal{P}(h,z)u\rA_{L^2}^2+\tfrac{1}{3}\lA u\rA_{L^2}^2.
$$

{\bf 4. End of the proof when $z$ is a real number.}

The Plancherel theorem implies that 
$$
\lA u\rA_{L^2}^2\le 
\lA g(hD)u\rA_{L^2}^2+\lA G_0u\rA_{L^2}^2+\lA G_\infty u\rA_{L^2}^2.
$$
Therefore, it follows from 
the previous $L^2$-estimates for $G_0u$ and $G_\infty u$ that
$$
\lA u\rA_{L^2}^2\le 
\lA g(hD)u\rA_{L^2}^2+2\lA 
\mathcal{P}(h,z)u\rA_{L^2}^2+\tfrac23 \lA u\rA_{L^2}^2,
$$
and hence we have
$$
\lA u\rA_{L^2}^2\le 
3\lA g(hD)u\rA_{L^2}^2+6\lA 
\mathcal{P}(h,z)u\rA_{L^2}^2.
$$
Now, it follows from \e{n0} and \e{N1} that  
\be\label{n1}
\lA g(hD)u\rA_{L^2}^2
\les 
h^{-1}\|{\mathcal P}(h,z)u\|_{L^2}\lA u\rA_{L^2} +  h^\nu\lA u\rA_{L^2}^2.
\ee
We conclude that
$$
\lA u\rA_{L^2}^2\les h^{-1}\|{\mathcal P}(h,z)u\|_{L^2}\lA u\rA_{L^2} + 
\lA 
\mathcal{P}(h,z)u\rA_{L^2}^2+h^\nu\lA u\rA_{L^2}^2.
$$
Consequently, taking $h$ small enough, 
we obtain
$$
\lA u\rA_{L^2}\les h^{-1}\|{\mathcal P}(h,z)u\|_{L^2}.
$$
This concludes the proof of the desired result~\e{n20} when $z\in (1-\delta, 1+\delta)$.

{\bf 5. End of the proof when $z$ is a complex number.} 
Lastly, if $z$ is a complex number, we notice that 
\[ \mathcal P(h,z) = \mathcal P(h,\Re z)+(\Im z)\left( \sqrt{h}\chi +\Im z -2 i \Re z \right). \]
Thus for any complex number $z$ such that
$|\Re z-1|\leq \delta$ and $|\Im z|\leq \delta h$, we have 
\[\begin{split} 
\|\mathcal P(h,z)u\|_{L^2}\geq & \|\mathcal P(h,\Re z)u\|_{L^2} - \delta h (\sqrt{h}\|\chi\|_{L^{\infty}}+\delta h+2(1+\delta) ) \|u\|_{L^2} \\
\geq & (C^{-1}-C\delta)h\|u\|_{L^2}.
\end{split}\]
Consequently, for all $\delta$ and $h$ small enough, 
the estimate \e{n20} still holds.
\end{proof}

\section{resolvent bounds for finitely degenerate damping}
\label{sec:degdamp}

Let $\mathcal P(h,z)$ be as in \eqref{eq: phz}.
We also introduce
\begin{equation*}
    P_{\pm}(\tau):=\pm D-i\tau \chi -\tau^2, \ \tau\in \CC,
\end{equation*}
and the corresponding rescaled operators
\begin{equation}
\label{eq: ppm-def}
    \Pr_{\pm}(h,z):=\pm hD-i\sqrt{h} z\chi-z^2.
\end{equation}

The goal of this section is to prove the second part of Theorem \ref{thm: resolvent-est}. We start by rewriting the resolvent bounds in the semiclassical scale.

\begin{prop}
\label{prop: finite-p}
Suppose $\chi\in C^{\infty}(\mathbb T)$ has finite degeneracy as in Definition \ref{def: nondeg-chi}. Then for any $0<\epsilon<\frac{1}{4N+2}$, there exists $h_0>0$, $\delta>0$, $C>0$, such that for $0<h<h_0$, $|\Re z-1|\leq \delta$, $|\Im z|\leq \delta h^{ \frac{4N+1}{4N+2}+\epsilon }$, such that for any $u\in C^{\infty}(\mathbb T)$, we have 
\begin{equation*}
    \|u\|_{L^2}\leq Ch^{ -\frac{4N+1}{4N+2}-\epsilon }\|\mathcal P(h,z)\|_{L^2}.
\end{equation*}
\end{prop}

We reduce the proof of the resolvent bound for $\mathcal P$ to the resolvent bounds for $\mathcal P_{\pm}$. The basic idea is that near the characteristic set of $\mathcal P$, the frequency is comparable to $\pm 1$, hence we can ``replace'' $\mathcal P$ by $\mathcal P_{\pm}$. Away from the characteristic set of $\mathcal P$, we have the ellipticity of $\mathcal P$. $\mathcal P_{\pm}$ is easier to analyze as it is an ordinary differential operator and we can use integrating factors to simplify $\mathcal P_{\pm}$.

\begin{figure}[b]
    \includegraphics[scale=0.6]{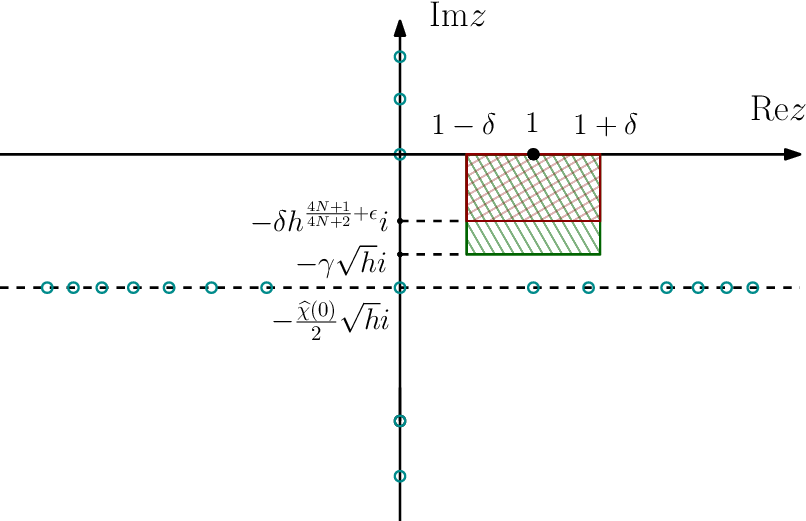}
    \caption{Resonances of $\mathcal P_{\pm}(h,z)$, and $z$-domains of estimations in Proposition \ref{prop: finite-p} (red) and Proposition \ref{prop: finite-pmp} (green).}
    \label{fig: ppm-res}
\end{figure}

\begin{prop}
\label{prop: finite-pmp}
Suppose $\chi$ has finite degeneracy as in Definition \ref{def: nondeg-chi} and $\Pr_{\pm}$ are as in \eqref{eq: ppm-def}. 
Then
    \begin{enumerate}[1.]
        \item For $h>0$, all resonances of $\mathcal P_{\pm}(h,z)$ lie on the lines 
        \[ \Re z=0 \text{ or } \Im z =-\frac{\widehat{\chi}(0)}{2}\sqrt{h}i.\]

        \item For any $0<\gamma<\frac{\widehat{\chi}(0)}{2}$, there exists $C>0$ such that for any $|\Re z-1|\leq \delta$, $\Im z\geq -\gamma\sqrt{h}$, and any $u\in C^{\infty}(\mathbb T)$, we have
        \begin{equation*}
            \|u\|_{L^2}\leq Ce^{\frac{C}{\sqrt{h}}}\|\mathcal P_{\pm}(h,z)u\|_{L^2}.
        \end{equation*}

        \item For any $0<\epsilon<\frac{1}{4N+2}$, there exists $h_0>0$, $\delta>0$, such that for $0<h<h_0$, $|\Re z-1|\leq \delta$, $|\Im z|\leq \delta h^{\frac{4N+1}{4N+2}+\epsilon }$, there exists $C>0$ such that for any $u\in C^{\infty}(\mathbb T)$, we have 
\begin{equation*}
    \|u\|_{L^2}\leq Ch^{ -\frac{4N+1}{4N+2}-\epsilon }\|\mathcal P_{\pm}(h,z)\|_{L^2}.
\end{equation*}
    \end{enumerate}
\end{prop}
\begin{proof}[Proof of Proposition \ref{prop: finite-p} using Proposition \ref{prop: finite-pmp}]

Let $G_0$, $G_{\infty}$, $G_{\pm}$ be as in the proof of Proposition \ref{prop: semi-loc}. Then for $z\in (1-\delta, 1+\delta)$, $u\in C^{\infty}(\mathbb T)$, we have 
    \[ G_0(h|D|-z^2)u = G_0\mathcal P(h,z)u-iz\sqrt{h}G_0 u. \]
    Thus for $\delta>0$ sufficiently small, using Plancherel theorem, we have 
    \begin{equation}
    \label{eq: g0-est}
    \|G_0 u\|_{L^2}\les \|\mathcal P(h,z)u\|_{L^2}+\sqrt{h}\|u\|_{L^2}. 
    \end{equation}
    A similar argument shows that 
    \begin{equation}\label{eq: gi-est} 
    \|G_{\infty} u\|_{L^2}\les \|\mathcal P(h,z)u\|_{L^2}+\sqrt{h}\|u\|_{L^2}. 
    \end{equation}

    Notice that we have the identity 
    \[ \mathcal P_{\pm}(h,z)G_{\pm} = G_{\pm}\mathcal P(h,z) - iz\sqrt{h}[\chi,G_{\pm}]. \]
    Use Proposition \ref{prop: finite-pmp} and we have 
    \[\begin{split} 
    \|G_{\pm}u\|_{L^2}
    \les & h^{-\frac{4N+1}{4N+2}-\epsilon}\|\mathcal P_{\pm}(h,z)G_{\pm}u\|_{L^2} \\
    \les & h^{-\frac{4N+1}{4N+2}-\epsilon}\|G_{\pm}\mathcal P(h,z)\|_{L^2}+h^{\frac{N+1}{2N+1} -\epsilon }\|h^{-1}[\chi, G_{\pm}]u\|_{L^2}. 
    \end{split}\]
    Since $h^{-1}[\chi,G_{\pm}]\in \Psi^{-1}(\mathbb T)$, we know $\|h^{-1}[\chi,G_{\pm}]\|_{L^2\to L^2}\leq C$. Therefore
    \begin{equation}
    \label{eq: gpm-est}
    \|G_{\pm}u\|_{L^2}\les h^{-\frac{4N+1}{4N+2}-\epsilon}\|\mathcal P(h,z)u\|_{L^2}+h^{\frac{N+1}{2N+1}-\epsilon}\|u\|_{L^2}. 
    \end{equation}
    Gathering estimates \eqref{eq: g0-est}, \eqref{eq: gi-est}, \eqref{eq: gpm-est}, we find that 
    \[ \|u\|_{L^2}\les h^{-\frac{4N+1}{4N+2}-\epsilon}\|\mathcal P(h,z)u\|_{L^2}+\sqrt h\|u\|_{L^2}. \]
    Thus when $h$ is sufficiently small, we have 
    \[ \|u\|_{L^2}\les h^{-\frac{4N+1}{4N+2}-\epsilon} \|\mathcal P(h,z)u\|_{L^2}. \]
    This proves Proposition \ref{prop: finite-p} when $z\in (1-\delta,1+\delta)$. 
    The proof of Proposition \ref{prop: finite-p} is completed by applying the same triangle inequality argument as in the last step of the proof of Proposition \ref{prop: semi-loc} when $z$ is complex.
\end{proof}

The rest of this section is devoted to proving Proposition \ref{prop: finite-pmp}, that is, the spectral gap and the resolvent bound for $\mathcal P_{\pm}$. The main idea is to consider a second microlocalization near zeros of $\chi$: away from the zeros of $\chi$, $\chi$ has lower bounds; near zeros of $\chi$, we use the smallness of the second microlocalization and the explicit Green's formula for $\mathcal P_{\pm}$.

We start by recording a property of $\chi$ when it has finite degeneracy.
\begin{lemm}
\label{lem: local-min}
    Let $\chi$ be as in Definition \ref{def: nondeg-chi}, then there exists $0<C_1<C_2$, $\delta>0$, such that 
    \begin{enumerate}[1.]
    \item For any $|x-x_k|<2\delta$, $1\leq k\leq n$, we have 
    \[ C_1\leq \frac{\chi(x)}{(x-x_k)^{2N_k}}\leq C_2, \ C_1\leq \frac{\chi^{\prime}(x)}{(x-x_k)^{2N_k-1}} \leq C_2.\]
    \item If $y\in \mathbb T$ is a local minimum of $\chi$ such that $\chi(y)\neq 0$, then we have 
    \[ \chi(y)>10\max_{x\in \cup_k[x_k-2\delta, x_k+2\delta]} \chi(x). \]
    \end{enumerate}
\end{lemm}
\begin{proof}
    The first conclusion follows from Taylor expansion
    \[ \chi(x)=\frac{1}{(2N_k)!}\chi^{(2N_k)} (x_k)(x-x_k)^{2N_k}+O(|x-x_k|^{2N_k+1}), \ |x-x_k|\to 0 \]
    and a similar expansion for $\chi^{\prime}$. 

    For the second claim, we notice that $\chi$ is a non-vanishing continuous function on $J:=\mathbb T\setminus \cup_{k} (x_k-2\delta,x_k+2\delta)$. Since $J$ is compact, there exists $c>0$ such that $\chi(x)>c$ for $x\in J$. Thus the second claim can be achieved by shrinking the value of $\delta$.
\end{proof}

\begin{proof}[Proof of Proposition \ref{prop: finite-pmp}]
    We only prove for $\mathcal P_+$, the proof for $\mathcal P_-$ is similar.

    {\bf 1. Resonances of $\mathcal P_{\pm}$.}

    To study the resonances of $P_+$, we introduce the integrating factor 
    \[ e^{\frac{z}{\sqrt{h}}\rho(x)}\in C^{\infty}(\mathbb T), \ \rho(x):=\int_0^x(\chi(y)-\widehat{\chi}(0))dy \in C^{\infty}(\mathbb T). \]
    One can check that 
    \begin{equation*}
    \mathcal P_+(h,z) = e^{-\frac{z}{\sqrt{h}} \rho } \left( hD-iz\sqrt{h}\widehat{\chi}(0)-z^2 \right) e^{\frac{z}{\sqrt{h}} \rho}. 
    \end{equation*}
    Notice that $hD$ has eigenvalues $hk\in \mathbb Z$, and $-iz\sqrt{h}\widehat{\chi}(0)-z^2$ is a constant function. Therefore, for $h>0$, the resonances $z$ for $\mathcal P_+(h,z)$ must satisfy 
    \[ z^2+iz\sqrt{h}\widehat{\chi}(0)-hk=0, \ k\in \mathbb Z. \]
    We can solve 
    \begin{equation*}\begin{gathered}
        z=\left( -\frac{\widehat{\chi}(0)}{2}i\pm \sqrt{k-\frac{\widehat{\chi}(0)^2}{4}} \right)\sqrt{h}, \ k\geq \frac{\widehat{\chi}(0)^2}{4}; \\
        z=\left( -\frac{\widehat{\chi}(0)}{2}\pm \sqrt{\frac{\widehat{\chi}(0)^2}{4} -k } \right)\sqrt{h} i, \ k < \frac{\widehat{\chi}(0)^2}{4}.
    \end{gathered}\end{equation*}
    Thus resonances of $\mathcal P$ lie in $\{\Re z=0\}\cup \{ \Im z=-\frac{\widehat{\chi}(0)}{2}\sqrt{h} \}$.

    {\bf 2. Estimates for $\mathcal P_{\pm}$ away from zeros of $\chi$.}

    It suffices to prove the resolvent bound for $z\in (1-\delta, 1+\delta)$. When $z$ is complex we apply the same triangle inequality argument as in the last step of the proof of Proposition \ref{prop: semi-loc}. 

    Let $x_k$, $1\leq k\leq N$, be the zeros of $\chi$. For $r>0$, we denote 
    \[ I_k(r):=[x_k-r, x_k+r], \ I(r):=\bigcup_{k}I_k(r). \]
    Let $\varphi\in C^{\infty}(\mathbb T;[0,1])$ such that 
    \begin{equation*}
    \supp \varphi \subset [-2,2], \ \varphi|_{[-1,1]}=1.
    \end{equation*}
    For $\beta>0$, we introduce cut-off functions $\psi_k, \varphi_0, \varphi_1\in C^{\infty}(\mathbb T)$ 
    \begin{equation*}\begin{gathered}
    \psi_{k}(x):= \varphi\left( \frac{x-x_{k}}{h^{\beta}} \right), \  \varphi_0(x):=\sum_{k}\psi_{k}(x), \
    \varphi_1(x):=1-\sum_{k}\varphi \left( \frac{2(x-x_{k})}{h^{\beta}} \right). 
    \end{gathered}\end{equation*}
    Functions $\varphi_{0}$, $\varphi_1$ satisfies the following conditions
    \begin{equation*}
        \supp \varphi_0\subset  I(2 h^{\beta}), \ \varphi_0|_{I(h^{\beta})}=1; \ \supp \varphi_1\subset \mathbb T\setminus I(h^{\beta}/2), \ \varphi_1|_{\mathbb T\setminus I(h^{\beta})}=1.
    \end{equation*}
    \begin{figure}[b]
    \includegraphics[scale=0.6]{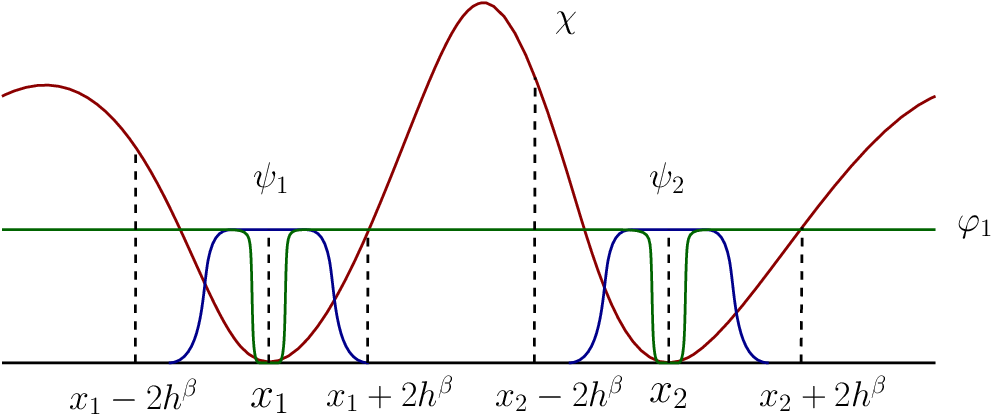}
    \caption{Second microlocalization near zeros of $\chi$.}
    \label{fig: secondmicro}
    \end{figure}
    
    We first consider the estimates for $\varphi_1 u$. Notice that 
    \[ \Im \langle \mathcal P_+(h,z)\varphi_1 u, \varphi_1 u \rangle =-z\sqrt{h} \int_{\mathbb T} \chi |\varphi_1 u|^2 dx \leq -Ch^{\frac12+2N\beta}\|\varphi_1 u\|_{L^2}^2.  \]
    Here we used the fact that 
    \[ \chi|_{\mathbb T\setminus I(h^{\beta}/2)}\geq ch^{2N\beta}, \ c>0, \ 0<h\ll 1, \]
    which follows from Lemma \ref{lem: local-min}.
    Now we find that
    \[\begin{split} 
    \|\varphi_1 u\|_{L^2}^2
    \leq & Ch^{-\frac12-2N\beta}|\langle \mathcal P_+(h,z)\varphi_1 u, \varphi_1 u \rangle| \\
    \leq & Ch^{-\frac12-2N\beta}\|\mathcal P_+(h,z)\varphi_1 u\|_{L^2}\|\varphi_1 u\|_{L^2}. 
    \end{split}\]
    Use Cauchy's inequality and we find 
    \[ \|\varphi_1 u\|_{L^2}\leq Ch^{-\frac12-2N\beta} \|\mathcal P_+(h,z)\varphi_1 u\|_{L^2}. \]
    Now notice that
    \[ \mathcal P_+(h,z)\varphi_1 u = \varphi_1 \mathcal P_+(h,z)u - ih \varphi_1^{\prime}u, \ |\varphi_1^{\prime}|\leq Ch^{-\beta}, \]
    and we conclude
    \begin{equation}
    \label{eq: p1-est}
        \|\varphi_1 u\|_{L^2}\leq C h^{-\frac12-2N\beta} \|\mathcal P_+(h,z)u\|_{L^2} + C h^{\frac12-(2N+1)\beta}\|u\|_{L^2}.
    \end{equation}

    {\bf 3. Estimates for $\mathcal P_{\pm}$ near zeros of $\chi$.}

    We now turn to estimate $\varphi_0 u$. We first solve 
    \[ \psi_k(x) u(x) = ih^{-1}\int_{x_k-2h^{\beta}}^x K(x,y)\mathcal P_+(h,z) \psi_k u(y) dy, \ K(x,y) = e^{-\frac{z}{\sqrt{h}}\int_{y}^x \chi(\theta) d\theta + \frac{i z^2}{h}(x-y) }. \]
    Since $\chi\geq 0$, we know $|K(x,y)|\leq 1$, when $x, y\in I_k(2h^{\beta})$, $y<x$. Hence 
    \[ |\psi_k u|\leq h^{-1}\int_{I_k(2h^{\beta})} |\mathcal P_+(h,z) \psi_k u(y) |dy \leq Ch^{-1+\frac{\beta}{2}}\|\mathcal P_+(h,z)\psi_ku\|_{L^2}. \]
    Integrate over $\mathbb T$ and notice that $\psi_k$ is supported in $I_k(2h^{\beta})$, and we find 
    \[ \|\psi_k u\|_{L^2}\leq Ch^{-1+\beta}\| \mathcal P_+(h,z)\psi_k u \|_{L^2}. \]
    Again notice 
    \[ \mathcal P_+(h,z)\psi_ku = \psi_k\mathcal P_+(h,z)u-ih\psi_k^{\prime} u, \ |\psi^{\prime}_k|\leq Ch^{-\beta}, \ \supp \psi_k^{\prime}\subset \{ \varphi_1=1\},  \]
    and we conclude that 
    \begin{equation*}
        \|\psi_ku\|_{L^2}\leq Ch^{-1+\beta}\|\mathcal P_+(h,z)u\|_{L^2}+C\|\varphi_1 u\|_{L^2}.
    \end{equation*}
    It follows by the definition of $\varphi_0$ now that 
    \begin{equation}
    \label{eq: p2-est}
        \|\varphi_0 u\|_{L^2}\leq Ch^{-1+\beta}\|\mathcal P_+(h,z)u\|_{L^2}+C\|\varphi_1 u\|_{L^2}.
    \end{equation}

    {\bf 4. End of the proof.}

    It remains to glue estimates \eqref{eq: p1-est} and \eqref{eq: p2-est} to obtain 
    \begin{equation*}
        \|u\|_{L^2}\leq C\max( h^{-1+\beta}, h^{-\frac12-2N\beta} )\|\mathcal P_+(h,z)u\|_{L^2} + Ch^{\frac12-(2N+1)\beta}\|u\|_{L^2}.
    \end{equation*}
    For $0<\epsilon<\frac{1}{4N+2}$, we take $\beta = \frac{1}{4N+2}-\epsilon>0$,
    then direct computations show that 
    \[\begin{gathered} 
    -1+\beta = -\tfrac{4N+1}{4N+2}-\epsilon < -\tfrac{4N+1}{4N+2}, \ -\tfrac12-2N\beta = -\tfrac{4N+1}{4N+2}+2N\epsilon>-\tfrac{4N+1}{4N+2}, \\ \tfrac12-(2N+1)\beta =  (2N+1)\epsilon>0.
    \end{gathered}\]
    Therefore we have 
    \begin{equation*}
        \|u\|_{L^2}\leq Ch^{-\frac{4N+1}{4N+2}-\epsilon}\|\mathcal P_+(h,z)u\|_{L^2}+Ch^{(2N+1)\epsilon}\|u\|_{L^2}.
    \end{equation*}
    For a fixed $\epsilon>0$, there exists $h_0=h_0(\epsilon)$, such that $Ch^{(2N+1)\epsilon}<\frac12$ when $0<h<h_0$, and in this situation we obtain
    \[ \|u\|_{L^2}\leq Ch^{-\frac{4N+1}{4N+2}-\epsilon}\|\mathcal P_+(h,z)u\|_{L^2}. \]
    This concludes the proof of Proposition \ref{prop: finite-pmp}.
\end{proof}

\section{Polynomial energy decay}
\label{ResToSG}

Here, for the benefit of the reader, we give a sketch proof of the equivalence between the resolvent estimates proved above and the corresponding energy and semigroup bounds for the damped wave equation.  This gives an overview of the works of \cite{anantharaman2014sharp,Batty-Duyckaerts} taking $A = |D|$ and $B = \sqrt{\chi}$, which builds on ideas developed in \cite{Lebeau96,LebeauRobbiano}.  Since we are working in a very explicit setting with relatively simple operators, connecting the resolvent estimate to the damped wave decay can be done in a fairly explicit manner. 

To start with, we rewrite the damped fractional wave equation as 
\begin{equation}
\label{eq: damp-sys}
    \partial_t U=\mathbf A U, \ U(0)=U_0, \ \mathbf A:=\begin{pmatrix} 0 & I \\ -|D| & -\chi \end{pmatrix}, \ U=\begin{pmatrix} u \\ \partial_t u \end{pmatrix},
\end{equation}
where $U_0=\begin{pmatrix} u_0 \\ u_1 \end{pmatrix}\in H^{\frac12}\times L^2$, $\mathbf A$ is an operator on $\mathscr H:=H^{\frac12}\times L^2$ with domain $\mathcal D(\mathbf A)=H^{1}\times H^{\frac12}$. The space $\mathscr H$ is equipped with a seminorm $\left|\begin{pmatrix} w_1 \\ w_2 \end{pmatrix}\right|_{\mathscr H}:=\||D|^{\frac12}w_1\|_{L^2}+\|w_2\|_{L^2}$. 

The following proposition established by Anantharaman--L\'eautaud \cite{anantharaman2014sharp} connects the resolvent estimates and energy decay rates of the damped fractional wave equation. 
\begin{prop}[{\cite[Proposition 2.4]{anantharaman2014sharp}}]
\label{prop: a-l}
Suppose $\chi\in L^{\infty}(\mathbb T)$. Let $P(\tau)$ be as in \eqref{eq: p-def}. Then for $\alpha>0$, the following statements are equivalent 
\begin{enumerate}[1.]
    \item[1.] There exists $C>0$ such that for any $(u_0, u_1)\in H^1(\mathbb T)\times H^{\frac12}(\mathbb T)$, there holds 
    \[ E(u,t)\leq \frac{C}{t^{2\alpha}}\left|\mathbf A \begin{pmatrix} u_0 \\ u_1 \end{pmatrix} \right|^2_{\mathscr H}. \]
    Notice that $ \left|\mathbf A \begin{pmatrix} u_0 \\ u_1 \end{pmatrix} \right|^2_{\mathscr H} \leq C\left(\||D|u_0\|_{L^2}^2+\|u_1\|_{H^{\frac12}}^2\right)\leq C\left( \|u_0\|_{H^1}^2+\|u_1\|_{H^{\frac12}}^2 \right) $.

    \item[2.] There exists $C>0$, such that for any $\tau\in \RR$, $|\tau|>C$, there holds
    \[ \|P(\tau)^{-1}\|_{L^2\to L^2}\leq C|\tau|^{\frac{1}{\alpha}-1}. \]
\end{enumerate}
\end{prop}

One would like to use the semi-group of $\mathbf A$ to solve the damped wave system \eqref{eq: damp-sys}, hence it is important to understand the spectrum of $\mathbf A$. As we see below, the set of resonances of $\mathbf A$ is the same set of resonances of $P(\tau)$. Hence it suffices to study the resonances of $P(\tau)$.
\begin{lemm}
\label{lem: mero}   
    Suppose $\chi\in L^{\infty}(\mathbb T)$ and $P(\tau)$ is as in \eqref{eq: p-def}. Then the resolvent 
    \begin{equation*}
    P(\tau)^{-1}: L^2(\mathbb T)\to H^1(\mathbb T), \ \tau\in \CC
    \end{equation*}
    is a meromorphic family of operators with finite rank poles.
    Moreover, $P(\tau)^{-1}$ is holomorphic in the region $\CC\setminus \mathcal O$, where
    \[ \mathcal O:=\left( \left[-\|\chi\|_{L^{\infty}}, 0 \right]i\right) \cup \left( \RR\setminus \{0\} + \left[-\tfrac12\|\chi\|_{L^{\infty}},0\right]i \right). \]
\end{lemm}
\begin{proof}
    Take $\tau_0\in \RR$, $\tau_0\neq 0$. We have the following resolvent identity 
    \begin{equation}
    \label{eq: rel-id}
        P(\tau) = (I-(i\tau\chi +\tau^2+\tau_0^2)(|D|+\tau_0^2)^{-1})(|D|+\tau_0^2).
    \end{equation}
    Notice that $|D|+\tau_0^2: H^1(\mathbb T)\to L^2(\mathbb T)$ is invertible. 
    Let 
    \[Q(\tau):=(i\tau\chi+\tau^2+\tau_0^2)(|D|+\tau_0^2)^{-1}: L^2(\mathbb T)\to L^2(\mathbb T), \tau\in\CC.\]
    Then $Q(\tau)$ is a family of compact operators, hence $I+Q(\tau)$ is a family of Fredholm operators. For $\tau_0\gg 1 $, we have 
    \begin{equation*}
        \|Q(i\tau_0)\|_{L^2(\mathbb T)\to L^2(\mathbb T)} \leq \|\chi\|_{L^{\infty}(\mathbb T)}/|\tau_0|<1.
    \end{equation*}
    Then for $\Im\tau\gg 1$, $I+Q(\tau)$ is a family of invertible Fredholm operators. The analytic Fredholm theory (see for instance \cite[Theorem C.8]{dyatlov2019mathematical}) implies that 
    \begin{equation*}
        (I+Q(\tau))^{-1}: L^2(\mathbb T)\to L^2(\mathbb T), \ \tau\in \CC
    \end{equation*}
    is a meromorphic family of operators with finite rank poles. This and \eqref{eq: rel-id} show that $P(\tau)^{-1}: L^2(\mathbb T)\to H^1(\mathbb T)$ is a meromorphic family of operators with finite rank poles.

    To see that $P(\tau)^{-1}$ is holomorphic in $\CC\setminus \mathcal O$, we write $\tau=\lambda+i\gamma$, $\lambda, \gamma\in \RR$, then
    \[ P(\lambda+i\gamma)=|D|-\lambda^2+\gamma(\chi+\gamma)-i\lambda (\chi+2\gamma). \]
    
    If $\lambda=0$, $\gamma>0$ or $\gamma<-\|\chi\|_{L^{\infty}}$, then there exists $c=c(\gamma)>0$ such that for $u\in H^1$, we have 
    \[ \langle P(i\gamma)u,u \rangle\geq \||D|^{\frac12}u\|^2_{L^2}+c\|u\|_{L^2}^2. \]
    Therefore $P(i\gamma)u=0$ implies $u=0$. Hence $\tau=\lambda+i\gamma$ can not be a resonance in this case.

    If $\lambda\neq 0$, $\gamma>0$ or $\gamma<-\frac12\|\chi\|_{L^{\infty}}$, then for $u\in H^1$, we compute 
    \[ \Im \langle P(\lambda+i\gamma)u, u \rangle = \lambda\langle (\chi+2\gamma)u,u \rangle \]
    Since either $\chi+2\gamma>0$ or $\chi+2\gamma<0$ for all $x\in \mathbb T$, we again see that $P(\lambda+i\gamma)u=0$ implies that $u=0$.
\end{proof}
The same result holds for $\mathbf A$. In fact,
\begin{lemm}
Let $\chi\in L^{\infty}(\mathbb T)$, $\mathbf A$ be as in \eqref{eq: damp-sys}. Then the resolvent 
\begin{equation*}
    (-i\tau I-\mathbf A)^{-1}: \mathscr H \to \mathcal D(\mathbf A), \ \tau\in \CC
\end{equation*}
is a meromorphic family of operators with finite rank poles and the poles (resonances) are exactly the resonances of $P(\tau)$. In particular, $0$ 
is a resonance of $\mathbf A$ and $\mathrm{Ker}_{\mathscr H}(\mathbf A)=\left\{\left. \begin{pmatrix} a \\ 0 \end{pmatrix} \ \right| \ a\in \CC \right\}$.
\end{lemm}
\begin{proof}
It suffices to recall the resolvent identity
\begin{equation}
    \label{LebId}
(-i\tau I - \mathbf A)^{-1} = \begin{pmatrix}
    P(\tau)^{-1} (\chi -i\tau I) & P(\tau)^{-1} \\
    P(\tau)^{-1} (-i\tau \chi  - \tau^2 I) & -i\tau P(\tau)^{-1}
\end{pmatrix}.
\end{equation}
from \cite{Lebeau96}.
\end{proof}

We now sketch the proof of Proposition \ref{prop: a-l} below. For further details of the proof, we refer to \cite[\S 4]{anantharaman2014sharp}.

\begin{proof}[Sketch proof of Proposition \ref{prop: a-l}]
Notice that $0$ is the only resonance of $\mathbf A$ with a nonnegative imaginary part -- which correspond to the nondecaying part in the solution. Hence it is natural to split the eigenspace of $\tau=0$. For that, we let $\Pi_0: \mathscr H\to \mathscr H$ be the orthogonal projection onto $\mathrm{Ker}_{\mathscr H}(\mathbf A)$. We define 
\[ \mathring{\mathscr H}:=(I-\Pi_0)\mathscr H, \ \|U\|_{\mathring{\mathscr H}}:=|U|_{\mathscr H}, \ \mathring{\mathbf A}:=\mathbf A|_{\mathring{\mathscr H}}. \]
Then the solution to the damped wave system \eqref{eq: damp-sys} can be expressed in terms of the semi-group of $\mathring{\mathbf A}$:
\begin{equation}
\label{eq: semi-solution}
    U(t)=e^{t\mathring{\mathbf A}}(I-\Pi_0)U_0+\Pi_0 U_0.
\end{equation}
Notice now that 
\[ E(u,t)=|U(t)|^2_{\mathscr{H}} = \|e^{t\mathring{\mathbf A}} (I-\Pi_0)U_0 \|^2_{\mathring{\mathscr H}},  \]
Thus the first statement in Proposition \ref{prop: a-l} is equivalent to the following semi-group bound: there exists $C>0$ such that
\begin{equation}
\label{eq: semi-bound}
    \| e^{t \mathring{\mathbf A}} \mathring{ \mathbf A }^{-1} \|_{\mathring{\mathscr H}\to \mathring{\mathscr H}}\leq \frac{C}{t^{\alpha}}, \ t>C.
\end{equation}
We now recall the result \cite[Theorem 2.4]{borichev2010optimal} by Borichev--Tomilov and realize that \eqref{eq: semi-bound} is equivalent to the following resolvent estimate for $\mathring{\mathbf A}$: there exists $C>0$ such that 
\begin{equation}
\label{eq: circ-a}
    \|(-i\tau I-\mathring{ \mathbf A })^{-1}\|_{\mathring{ \mathscr H }\to \mathring{ \mathscr H } } \leq C|\tau|^{\frac{1}{\alpha}}, \ |\tau|>C.
\end{equation}

Use the following identity on $\mathscr H$ for $\tau\neq 0$
\[ (-i\tau I-\mathring{\mathbf A})^{-1}(I-\Pi_0) = (-i\tau I - \mathbf A)^{-1}(I-\Pi_0) = (-i\tau I-\mathbf A)^{-1} +\frac{\Pi_0}{i\tau} \]
and we find that
\begin{equation*}
    \left| \|(-i\tau I-\mathring{\mathbf A})^{-1}\|_{\mathring{\mathscr H} \to \mathring{ \mathscr H } } - \| (-i\tau I - \mathbf A )^{-1} \|_{\mathscr H\to \mathscr H}   \right|\leq \frac{C}{|\tau|}, \ |\tau|>1.
\end{equation*}
Thus \eqref{eq: circ-a} is equivalent to the bound for the resolvent of $\mathbf A$: there exists $C>0$ such that 
\begin{equation}
\label{eq: a-res}
    \|(-i\tau I-\mathbf A)^{-1}\|_{\mathscr H\to \mathscr H}\leq C |\tau|^{\frac{1}{\alpha}}, \ |\tau|>C.
\end{equation}

It remains to use \eqref{LebId} again to build the equivalence between \eqref{eq: a-res} and the second statement of Proposition \ref{prop: a-l}.
\end{proof}

\Remark
Iterating \eqref{eq: semi-bound} (see also \cite{Batty-Duyckaerts}), we have for $k\geq 1$,
\begin{equation*}
    \|e^{t\mathring{\mathbf A}} \mathring{\mathbf A}^{-k} \|_{\mathring{\mathscr H}\to \mathring{\mathscr H}} = \left\|\left( e^{\frac{t}{k} \mathring{\mathbf A} } \mathring{\mathbf A}^{-1} \right)^k\right\|_{\mathring{\mathscr H}\to \mathring{\mathscr H}} \leq \|e^{\frac{t}{k} \mathring{\mathbf A} }\mathring{\mathbf A}^{-1} \|^k_{\mathring{\mathscr H}\to \mathring{\mathscr H}} \leq \frac{ C^k k^{k\alpha} }{ t^{k\alpha}}.
\end{equation*}
Hence if we further assume $\chi\in C^{\frac{k}{2}}(\mathbb T)$ and $(u_0, u_1)\in H^{\frac{k+1}{2}}(\mathbb T) \times H^{\frac{k}{2}}(\mathbb T)$, then \eqref{eq: semi-solution} implies that 
\begin{equation*}
    E(u,t)\leq \frac{C_k}{t^{2k\alpha}} \left\|\mathring{\mathbf A}^k \begin{pmatrix} u_0 \\ u_1 \end{pmatrix}\right\|_{\mathring{\mathscr H}} \leq \frac{C_k}{ t^{2k\alpha}} \left( \|u_0\|_{H^{\frac{k+1}{2}}}^2 + \|u_1\|^2_{H^{\frac{k}{2}}} \right).
\end{equation*}
This means that the more regular the initial data is, the faster the energy decays.

\section{asymptotics of the resonances for small damping}
\label{sec:resdist}

This section is devoted to proving Theorem \ref{thm: resonance}. Since the first statement is a direct result of the resolvent bounds in Theorem \ref{thm: resolvent-est}, we focus on the proof of the second statement. The main tools we use are Grushin problems -- we refer to \cite{sjzw2007grushin} for an introduction and applications of Grushin problems.

Recall the notation
\begin{equation*}
    P(\nu,\tau)=|D|-i\nu\tau \chi-\tau^2: L^2(\mathbb T) \to L^2(\mathbb T), \ \nu>0.
\end{equation*}
The operator $P(0,\tau)$ has resonances $\pm \sqrt{k}$, $k\in \mathbb Z$, $k\geq 0$. Notice that if $\tau$ is a resonance of $P(\nu,\bullet)$, then $-\overline{\tau}$ is also a resonance of $P(\nu,\bullet)$. Thus we only need to consider the resonances $\tau$ with $\Re\tau\geq 0$.

\begin{proof}[Proof of the second statement of Theorem \ref{thm: resonance}]
For $k\in \mathbb Z$, $k>0$, we propose the following Grushin problem for $(\nu,\tau)$ in a neighborhood of $(0,\sqrt{k})\in \CC^2$:
\begin{equation*}
    \mathscr P(\nu,\tau):= 
    \begin{pmatrix}  P(\nu,\tau) & R_- \\ R_+ & 0 \end{pmatrix}: H^1(\mathbb T)\times \CC^2 \to L^2(\mathbb T)\times \CC^2,
\end{equation*}
where $R_{\pm}$ are given by 
\begin{equation*}\begin{gathered}
    R_-: \CC^2\to L^2(\mathbb T), \ R_-(a_1, a_1):=a_1 e^{ikx}+ a_2 e^{-ikx}, \\
    R_+: H^1(\mathbb T)\to \CC^2, \ R_+(v):=(\langle v,e^{ikx} \rangle, \langle v,e^{-ikx} \rangle).
\end{gathered}\end{equation*}
Here the inner product on $L^2(\mathbb T)$ is defined by 
\begin{equation*}
    \langle u,v \rangle:=\frac{1}{2\pi}\int_{\mathbb T} u(x)\overline{v(x)} dx.
\end{equation*}
For $k>0$, let $M_k(\tau)$ be the operator defined by 
\begin{equation*}
    M_k(\tau):=\Pi_{\perp}P(0,\tau)^{-1}\Pi_{\perp}: L^2(\mathbb T)\to L^2(\mathbb T), \ \tau \in \CC\setminus \{\pm \sqrt{\ell}\ | \ \ell\in \mathbb Z, \ \ell\geq 0 \}.
\end{equation*}
where $\Pi_{\perp}: L^2(\mathbb T)\to L^2(\mathbb T)$ is the orthogonal projection onto the orthogonal complement of $\mathrm{Ker}_{L^2}(P(0,\sqrt{k}))=\mathrm{span}(e^{ikx}, e^{-ikx})$. We have the following explicit formula for $M_k(\tau)$:
\begin{equation}
\label{eq: m-def}
    M_k(\tau)u(x)= \sum_{n\neq \pm k} \tfrac{a_n}{|n|-\tau^2}e^{inx}, \ \text{ when } u(x)=\sum_{n\in \mathbb Z} a_n e^{inx}.
\end{equation}
From \eqref{eq: m-def} we know $M_k(\tau)$ is in fact defined and holomorphic in a neighborhood of $\tau=\sqrt{k}$. Moreover, for $\tau$ near $\sqrt{k}$, $M_k(\tau): L^2(\mathbb T)\to H^1(\mathbb T)$ is a bounded operator. Let $\Pi_k$ be the orthogonal projection onto $\mathrm{Ker}_{L^2}(P(0,\sqrt{k}))$, then we have 
\begin{equation*}
    P(0,\tau)^{-1}=M_k(\tau)+\frac{\Pi_k}{k-\tau^2}.
\end{equation*}

A direct computation then shows that $\mathscr P(0,\tau)$ has an inverse 
\begin{equation*}
    \mathscr E(\tau):=\begin{pmatrix} M_k(\tau) & E_+ \\ E_- & E_{-+}(\tau)\end{pmatrix}: L^2(\mathbb T)\times \CC^2 \to H^1(\mathbb T)\times \CC^2,
\end{equation*}
where $E_{-+}(\tau)=(k-\tau^2)I_2$ and $E_{\pm}$ are given by 
\begin{equation*}\begin{gathered}
    E_+: \CC^2\to H^1(\mathbb T), \ E_+(a_1,a_2):=a_1 e^{ikx}+a_2 e^{-ikx}, \\
    E_-: L^2(\mathbb T) \to \CC^2, \ E_-(v):=(\langle v,e^{ikx} \rangle, \langle v,e^{-ikx} \rangle).
\end{gathered}\end{equation*}
Notice that
\[ P(\nu,\tau)-P(0,\sqrt{k}) = -i\nu\tau\chi+k-\tau^2. \]
Hence for $(\nu,\tau)$ in an open neighborhood of  $(0,\sqrt{k})\in \CC^2$, there holds
\begin{equation*}
    \max_{\nu,\tau}
    \begin{Bmatrix}
    \|M_k(\tau)(P(\nu,\tau)-P(0,\sqrt{k}))\|_{H^1\to H^1} \\
    \|(P(\nu,\tau)-P(0,\sqrt{k})M_k(\tau)\|_{L^2\to L^2} 
    \end{Bmatrix}
    <\frac12.
\end{equation*}
Now by \cite[Lemma C.3]{dyatlov2019mathematical}, we know $\mathscr P(\nu,\tau)$ is invertible for $(\nu,\tau)$ in a neighborhood of $(0,\sqrt{k})\in \CC^2$. We denote the inverse of $\mathscr P(\nu,\tau)$ by 
\begin{equation*}
    \mathscr E(\nu,\tau):=\begin{pmatrix} E(\nu,\tau) & E_+(\nu,\tau) \\ E_-(\nu,\tau) & E_{-+}(\nu,\tau) \end{pmatrix}: L^2(\mathbb T)\times \CC^2 \to H^1(\mathbb T)\times \CC^2.
\end{equation*}
Then $\mathscr E(\nu,\tau)$ is holomorphic near $(0,\sqrt{k})$ and $\mathscr E(0,\tau)=\mathscr E(\tau)$.
Recall that the invertibility of $P(\nu,\tau): H^1(\mathbb T)\to L^2(\mathbb T)$ is equivalent to the invertibility of $E_{-+}(\nu,\tau)$. In fact, there holds the following Schur complement formula
\[ P^{-1}=E-E_+ E_{-+}^{-1}E_-, \ E_{-+}^{-1}=-R_+P^{-1}R_-. \]

Since $E_{-+}(\nu,\tau)$ is a $2\times 2$ matrix, its invertibility is much easier to characterize. 
Let $\mathcal L(\nu,\tau):=\det E_{-+}(\nu,\tau)$, then we know $\mathcal L$ is holomorphic near $(0,\sqrt{k})$ and $\tau\in \mathscr R(\nu)$ if and only if $\mathcal L(\nu,\tau)=0$. Since $E_{-+}(0,\tau)=(k-\tau^2)I_1$, we know $\mathcal L(0,\tau)=(k-\tau^2)^2$. By the Weierstrass Preparation Theorem, we know that for $(\nu,\tau)$ in an open neighborhood $U_k$ of $(0,\sqrt{k})\in \CC^2$, there exist holomorphic functions $g_0(\nu)$, $g_1(\nu)$ and $\mathcal N(\nu,\tau)$ such that 
\begin{equation*}\begin{gathered}
    \mathcal L(\nu,\tau)=\mathcal N(\nu,\tau)( (\tau-\sqrt{k})^2+g_0(\nu) (\tau-\sqrt{k}) +g_1(\nu) ), \\
    g_0(0)=g_1(0)=0, \ \mathcal N(\nu,\tau)\neq 0, \ (\nu,\tau)\in U_k.
\end{gathered}\end{equation*}
Thus in $U_k$, the zeros $(\nu,\tau(\nu))$ of $\mathcal L(\nu,\tau)$ are
\begin{equation*}
    \tau_{k,\pm}(\nu)=\sqrt{k}-\frac{g_0(\nu)\pm \sqrt{g_0(\nu)^2-4g_1(\nu)}}{2}.
\end{equation*}
Since $g_0$, $g_1$ are holomorphic functions, we have the Taylor expansion 
\begin{equation}
\label{eq: taylor}
g_0(\nu)^2-4g_1(\nu) = \nu^{r}(a_0+a_1 \nu+\cdots), \ r\in \mathbb Z, \ r>0, \ a_0\neq 0. 
\end{equation}
Thus one of the following statement is true: 
\begin{enumerate}[(i)]
    \item[1.] Either both $\tau_{k,\pm}$ are holomorphic functions of $\nu$ for $\nu$ near $0\in \CC$, that is,
    \begin{equation} 
    \label{eq: ser1}
    \tau_{k,\pm}(\nu) = \sqrt{k}+\sum_{\ell=1}^{\infty} a_{\ell,\pm} \nu^{\ell}, \ a_{\ell,\pm}\in\CC;
    \end{equation}

    \item[2.] Or both $\tau_{k,\pm}$ have power series expansion in terms of $\sqrt{\nu}$ when $\nu$ is near $0$ and $\nu\neq 0$: 
    \begin{equation} 
    \label{eq: ser2}
    \tau_{k,\pm}(\nu) = \sqrt{k} + \sum_{\ell=1}^{\infty} b_{\ell} (\pm\sqrt{\nu})^{\ell}, \ b_{\ell}\in \CC. 
    \end{equation}
\end{enumerate}

On the other hand, we have the following expansion of $E_{-+}$ (see for instance \cite[Lemma C.3]{dyatlov2019mathematical}):
\begin{equation}\label{eq: emp}
    \begin{gathered}
    E_{-+}(\nu,\tau) = (k-\tau^2)I_2 - \sum_{\ell=1}^{\infty}(i\nu \tau)^{\ell}  T_{\ell}(\tau), \\
    T_{\ell}(\tau):=\begin{pmatrix} 
        \langle \chi(M_k(\tau)\chi)^{\ell-1} e^{ikx}, e^{ikx} \rangle 
        & \langle \chi (M_k(\tau)\chi )^{\ell-1} e^{ikx}, e^{-ikx} \rangle \\ 
        \langle \chi (M_k(\tau)\chi )^{\ell-1} e^{-ikx}, e^{ikx} \rangle 
        & \langle \chi (M_k(\tau)\chi)^{\ell-1} e^{-ikx}, e^{-ikx} \rangle
    \end{pmatrix}.
    \end{gathered}
\end{equation}
In particular, we know $\nu^{-1}E_{-+}(\nu,\sqrt{k})$ is holomorphic near $\nu=0$. This indicates that $\nu^{-2}\mathcal L(\nu,\sqrt{k}) = \nu^{-2}\mathcal N(\nu,\sqrt{k}) g_1(\nu) $ is holomorphic. Therefore $\nu^{-2}g_1(\nu)$ is holomorphic near $\nu=0$ and in the expansion \eqref{eq: taylor}, we must have $r\geq 2$. As a result, in either \eqref{eq: ser1} or \eqref{eq: ser2}, we must have 
\begin{equation}
\label{eq: tpm}
    \tau_{k,\pm}(\nu) = \sqrt{k}+ {c_{\pm}} \nu + o(\nu). 
\end{equation}
Inserting \eqref{eq: tpm} in \eqref{eq: emp} gives
\begin{equation*}
    E_{-+}(\nu,\tau_{k,\pm})=(-2\sqrt{k} c_{\pm} \nu +o(\nu))I_2 - (i\sqrt{k}\nu +o(\nu))
    \begin{pmatrix} \widehat{\chi}(0) & \overline{\widehat{\chi}(2k)} \\ \widehat{\chi}(2k) & \widehat{\chi}(0) \end{pmatrix} + S(\nu), 
\end{equation*}
where $S(\nu)$ is a $2\times 2$ matrix and each entry of $S(\nu)$ is of $o(\nu)$. Use the fact that $\mathcal L(\nu, \tau_{k,\pm}(\nu))=0$ and we find 
\begin{equation*}
    \det\left( 2c_{\pm}  I_2 + i \begin{pmatrix} \widehat{\chi}(0) & \overline{\widehat{\chi}(2k)} \\ \widehat{\chi}(2k) & \widehat{\chi}(0) \end{pmatrix} \right)=0,
\end{equation*}
from which we solve 
\begin{equation*}
    c_{\pm}=-\frac{\widehat{\chi}(0)\pm |\widehat{\chi}(2k)|}{2}i.
\end{equation*}
This completes the proof.
\end{proof}

\Remarks 
1. The expansion of $\tau_{k,\pm}$ as in \eqref{eq: ser1}, \eqref{eq: ser2} are special cases of Puiseux series expansion. See \cite[\S 2.4]{wang2022zero} and the references there for a brief introduction to Puiseux series and their applications to spectral problems. 

\noindent 
2. By inserting \eqref{eq: tpm} in \eqref{eq: emp}, one can actually obtain successively the coefficient for any order terms in the expansions \eqref{eq: ser1} or \eqref{eq: ser2}.

\section{Numerical Simulations}
\label{sec:numsims}

We run our numerical approximations on three models of damping to highlight the theoretical results above.  The three damping functions we consider are:
\begin{align*}
\chi_1 (x) & = \frac14 \cos^2(x) ,  \quad \ \chi_2 (x) = e^{-2 (x-\pi)^2} , \label{chi:gauss} \\
\chi_3 (x) & = \frac18 \left[ \tanh \left( 20 \left( x - \pi + \frac14 \right) \right) - \tanh \left( 20 \left( x - \pi - \frac14 \right) \right) \right] ,
\end{align*}
where the function $\chi_1$ is a low-frequency damping, $\chi_2$ is a Gaussian damping with exponential decay at high frequency, and $\chi_3$ is a compactly supported damping function with slowly decaying Fourier modes.  

The equations to \eqref{eqn:dampwave} are rather straightforward to discretize as ODE systems in Fourier space, all the time dependent solvers are run by using Fourier pseudospectral methods in space with $2^9$ spatial grid points and integrating in time using the \texttt{ode45} package in \texttt{MATLAB} with relative and absolute tolerance levels set at $10^{-9}$ to ensure high accuracy of the solutions. 

\subsection{Demonstration of Polynomial Decay rates}

We can observe that decay rates for the damped fractional wave dynamics appear to be numerically very close to the polynomial rates predicted by our estimates for well-chosen initial data.  Indeed, in Figure \ref{fig:t2}, we observe relative decay properties comparing to the polynomial rate $t^{-2}$ by plotting the evolution of $t^2 E(t)$ for \eqref{eqn:dampwave} using initial conditions that are high frequency and localized far from the damping.  In particular, we take
\[
U(0,x) = \chi_3(x+\pi) \cos(10x)
\]
and observe that the decay rates are quite remarkably close to our predicted polynomial rates.

  To highlight this polynomial behavior, in Figure \ref{fig:t4} we plot $\log E/ \log t$ for both $\chi_1$ damping and $\chi_3$ damping with the same initial condition as above.  For the low frequency damping, $\chi_1$, we observe similar decay rates to those predicted in Theorem \ref{thm: decay}, namely close to $t^{-3}$.

\begin{figure}[htbp] 
   \centering
   \includegraphics[width=0.4\textwidth]{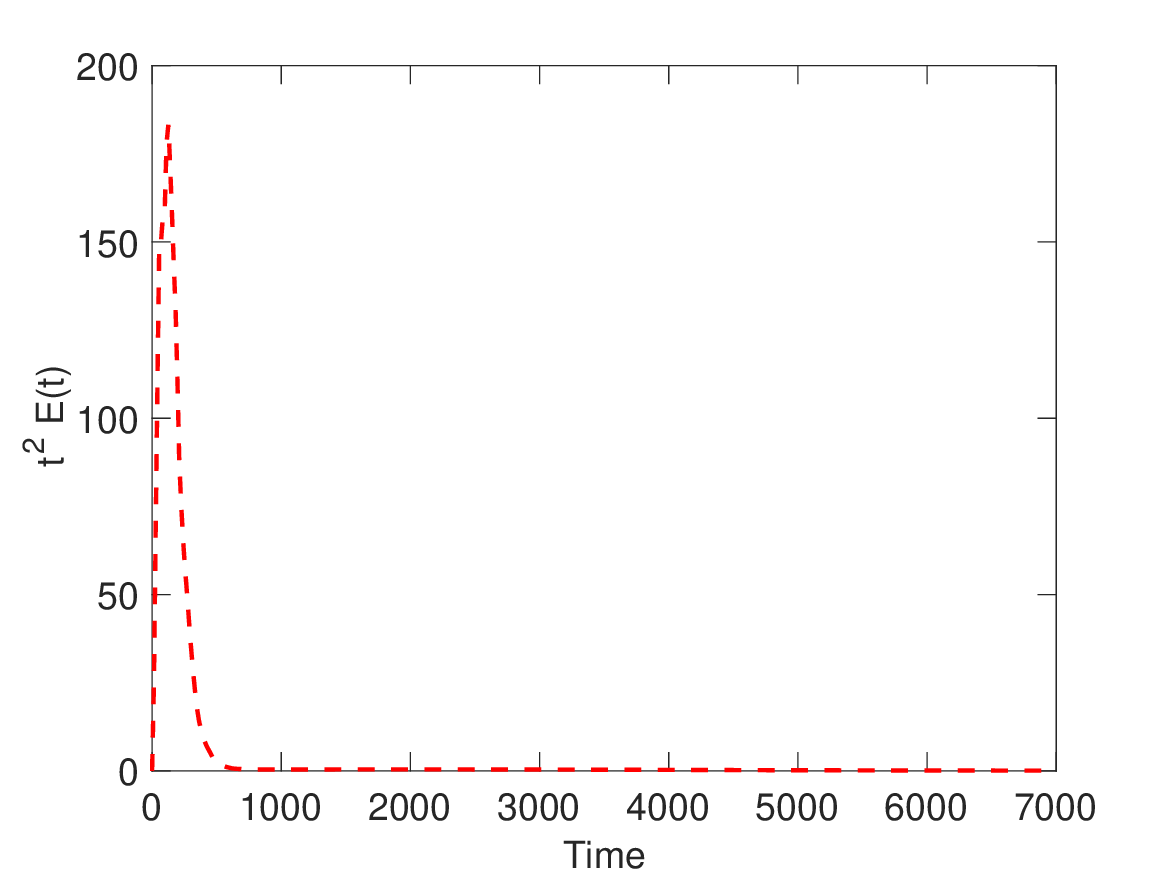} 
     \includegraphics[width=0.4\textwidth]{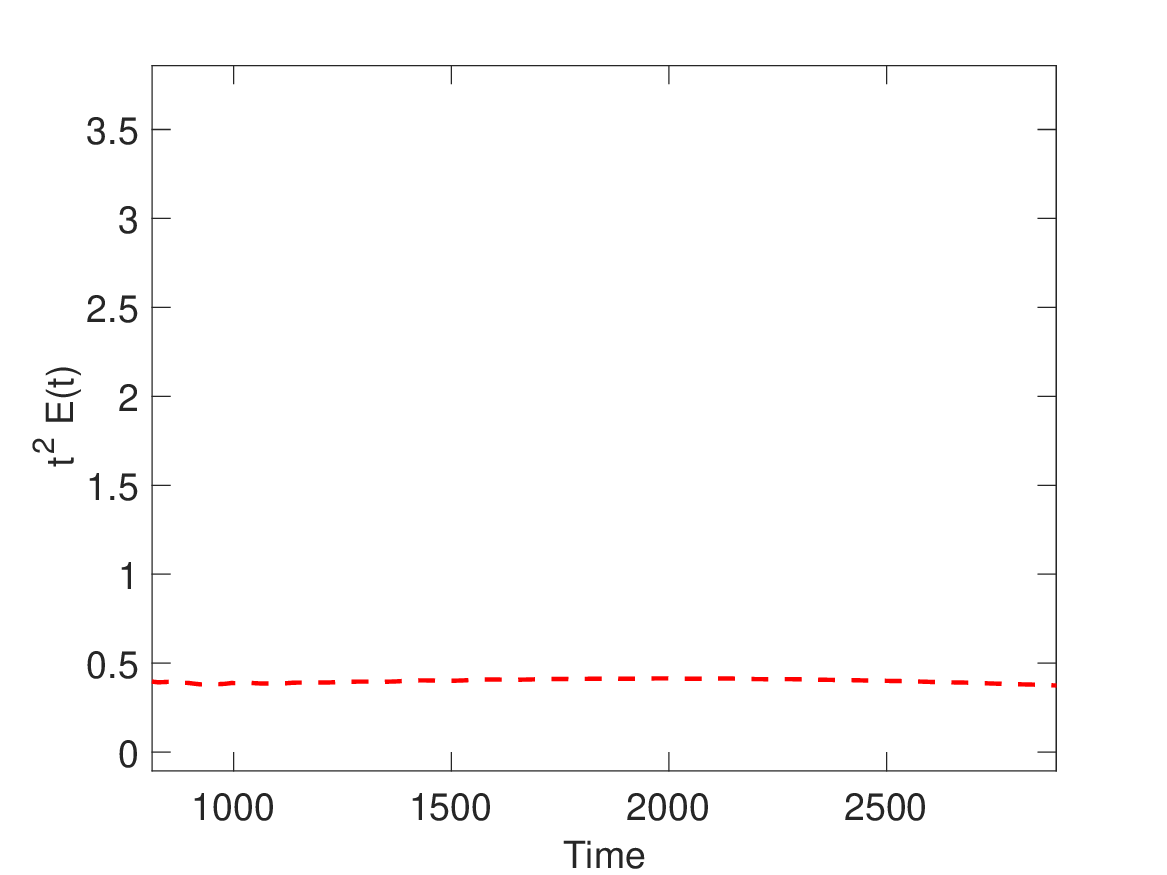} 
   \caption{\textbf{(Left)} A plot of $t^2 E(t)$ in the setting of $\chi_3$ damping and \textbf{(Right)} a zoom in near the large $t$ asymptotics for $U(0,x) = \chi_3(x+\pi) \cos(10x)$.}
   \label{fig:t2}
\end{figure}

\begin{figure}[htbp] 
   \centering
         \includegraphics[width=0.4\textwidth]{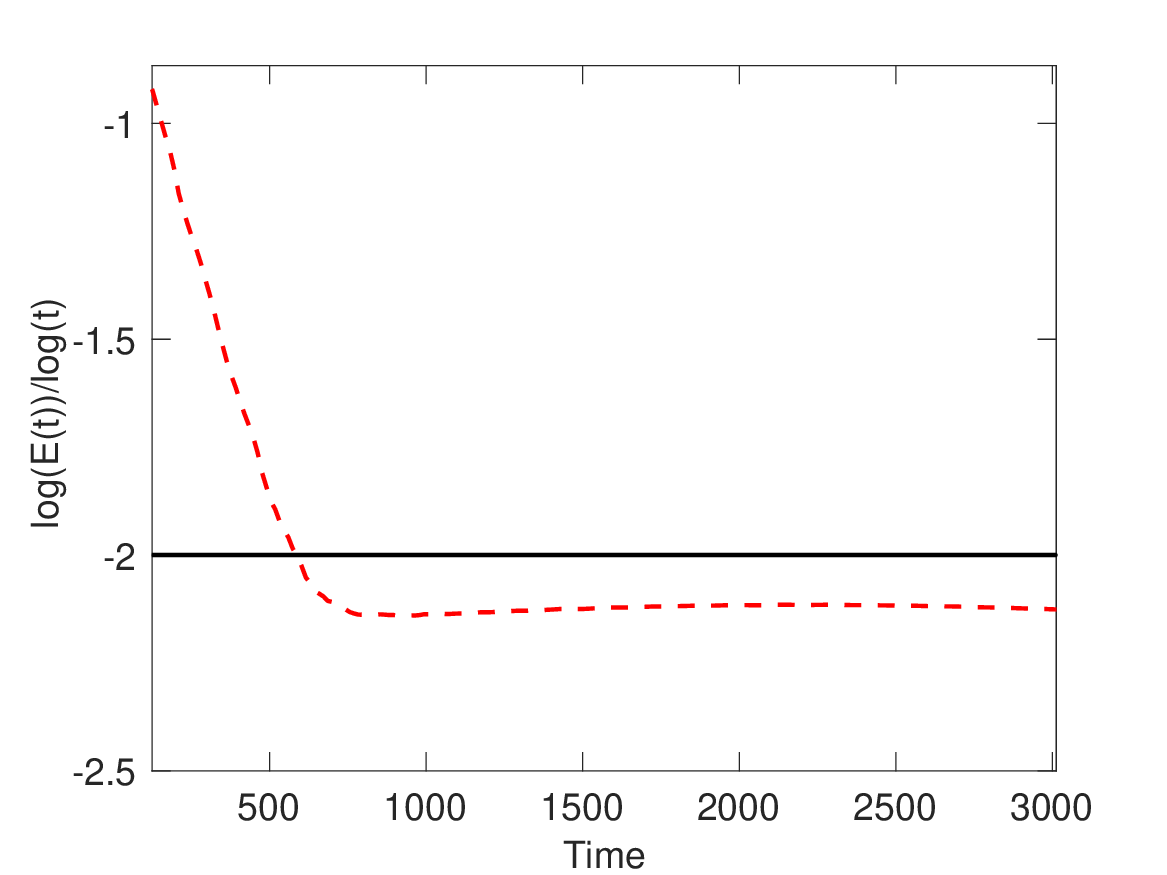} 
   \includegraphics[width=0.4\textwidth]{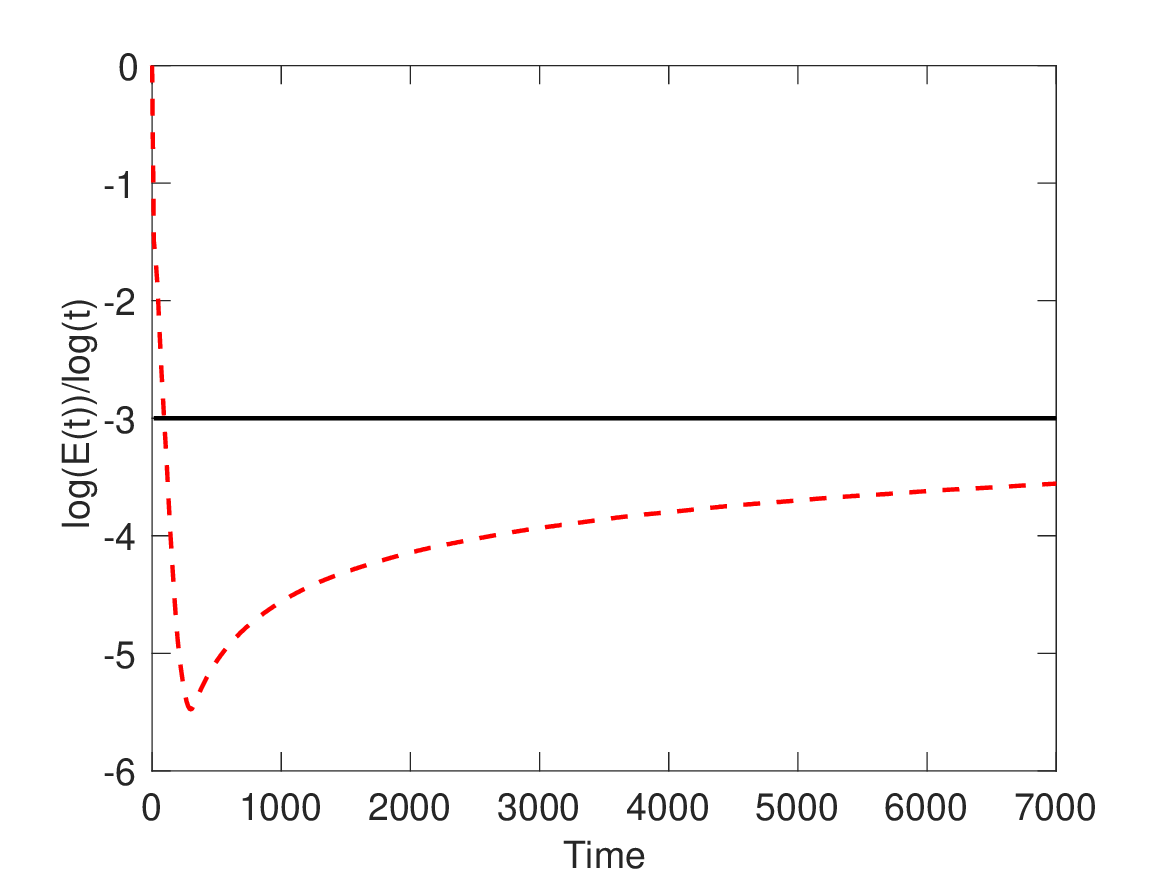} 
   \caption{\textbf{(Left)} The numerically observed values of $\log(E(t))/\log(t)$ demonstrating that the rate of decay is very close to $t^{-2}$ for $\chi_3$ damping.  \textbf{(Right)} A plot of $\log(E(t))/\log(t)$ in the setting of $\chi_1$ damping for demonstrating that the rate of decay is approaching $t^{-3}$.  In both, we have taken $U(0,x) = \chi_3(x+\pi) \cos(10x)$.}
   \label{fig:t4}
\end{figure}

\subsection{Approximation of the resonances at low frequency}

To begin, by implementing a low frequency approximation scheme, we can analyze the behavior of the resonances at low energy. For that we define 
\[ L_N:=\mathrm{span}_{\CC}\{ e^{ikx} \ | \ |k|\leq N \}, \ \pi_N: L^2(\mathbb T)\to L_N \text{ is the orthogonal projection. } \]
We introduce the discretization $P_N(\tau)$ of $P(\tau)$ using Fourier modes:
\[ P_N(\tau):= |D|-i\tau \pi_N \chi-\tau^2: L_N\to L_N. \]
Since $L_N$ is a linear space of dimension $2N+1$, we know for each $\tau\in \CC$, $P_N(\tau)$ is a $(2N+1)\times (2N+1)$ matrix. 
Indeed, if we expand $\chi$ using Fourier series
\[ \chi(x) = \sum_{n\in \mathbb Z} \widehat{\chi}(n) e^{inx}, \]
then using the basis $\{ e^{ikx} \}_{|k|\leq N}$ of $L_N$, $P_N(\tau)$ becomes a matrix
\[  P_N (\tau) =\mathrm{diag}(|n|)_{|n|\leq N} -(\tau^2+i\widehat{\chi}(0)\tau)I_{2N+1}-i\tau \sum_{n=1}^{2N}\left( \widehat{\chi}(n) J_{2N+1}^n + \overline{\widehat{\chi}(n)}(J_{2N+1}^t)^n \right), \]
where $J_{2N+1}$ is the $(2N+1)\times (2N+1)$ upper-triangular Jordan block.

Because $P(\tau)$ is a quadratic polynomial of $\tau$ with matrix coefficients, we know $\det(P_N(\tau))$ is polynomial of $\tau$ of order $2(2N+1)$. Therefore $\det(P_N(\tau))^{-1}$ is meromorphic function with $2(2N+1)$ poles (counting multiplicities). As a result, the inverse matrix
\[ P_N(\tau)^{-1}: L_N\to L_N, \ \tau\in \CC \]
is a meromorphic family of $(2N+1)\times (2N+1)$ matrices with $2(2N+1)$ poles (counting multiplicities). We denote the set of poles of $P_N(\tau)$, i.e., zeros of $\det(P_N(\tau))$, by $\mathscr R_N$. We regard $\mathscr R_N$ as a low frequency approximation to $\mathscr R$.

For a given damping function $\chi$, we can thus numerically construct function ${\rm det}P_N (\tau)$ whose form we can compute symbolically in \texttt{MATLAB}.  We can then find the roots of this polynomial equation in $\tau$.

We uniformly take $N=12$ as an approximation, and observe the following each of our three experiments.  Depending upon the potential, we have a number of resonances concentrating around $\mathrm{Im} (z) = \widehat{\chi}_j (0)$ for $j = 1,2,3$ as predicted by the asymptotics in \S \ref{sec:resdist}.  The observed resonances are computed using $N=12$.  The resonance with smallest imaginary part is $\mathrm{Im} (z) = -.0312,-.0378,-4.48 \times 10^{-4}$ for $\chi_j$ with $j=1,2,3$ respectively.  We observed that these resonances were stable under refinement of the approximation parameter $N$.  Of course, there is $0$ resonance for every example corresponding to the constant solution. 

We can compare our asymptotics from \S \ref{sec:resdist} to the computed approximate resonances with $N=12$ for a variety of damping functions and parameters $\nu$.  We consider again three experiments again motivated by low frequency damping ($\chi_1$), Gaussian damping ($\chi_2$) and compactly supported damping ($\chi_3$), but with varying amplitude depending upon the parameter $\nu$:
\begin{align*}
\chi_{1,\nu} (x) & = \nu \cos^2(x) ,  \ \ \chi_{2,\nu} (x) =  \nu e^{-2 (x-\pi)^2} ,  \\
\chi_{3,\nu} (x) & = \frac{\nu}{2} \left[ \tanh \left( 20 \left( x - \pi + \frac14 \right) \right) - \tanh \left( 20 \left( x - \pi - \frac14 \right) \right) \right] .
\end{align*}
In particular, our experiments above corresponded to $\nu = .25, 1, .25$ for $j=1,2,3$ respectively.  Figures \ref{fig:res4}-\ref{fig:res6} demonstrate quite well that our asymptotics remain quite robust for each of these potentials for small $\nu$ and still give a fair amount of insight especially into the real part of the resonances even for large $\nu$.  These resonances give insight into how states that are low-frequency and overlapping with the damping function decay in a significant fashion under the evolution of \eqref{eqn:dampwave}.

\begin{figure}[htbp] 
   \centering
      \includegraphics[width=0.3\textwidth]{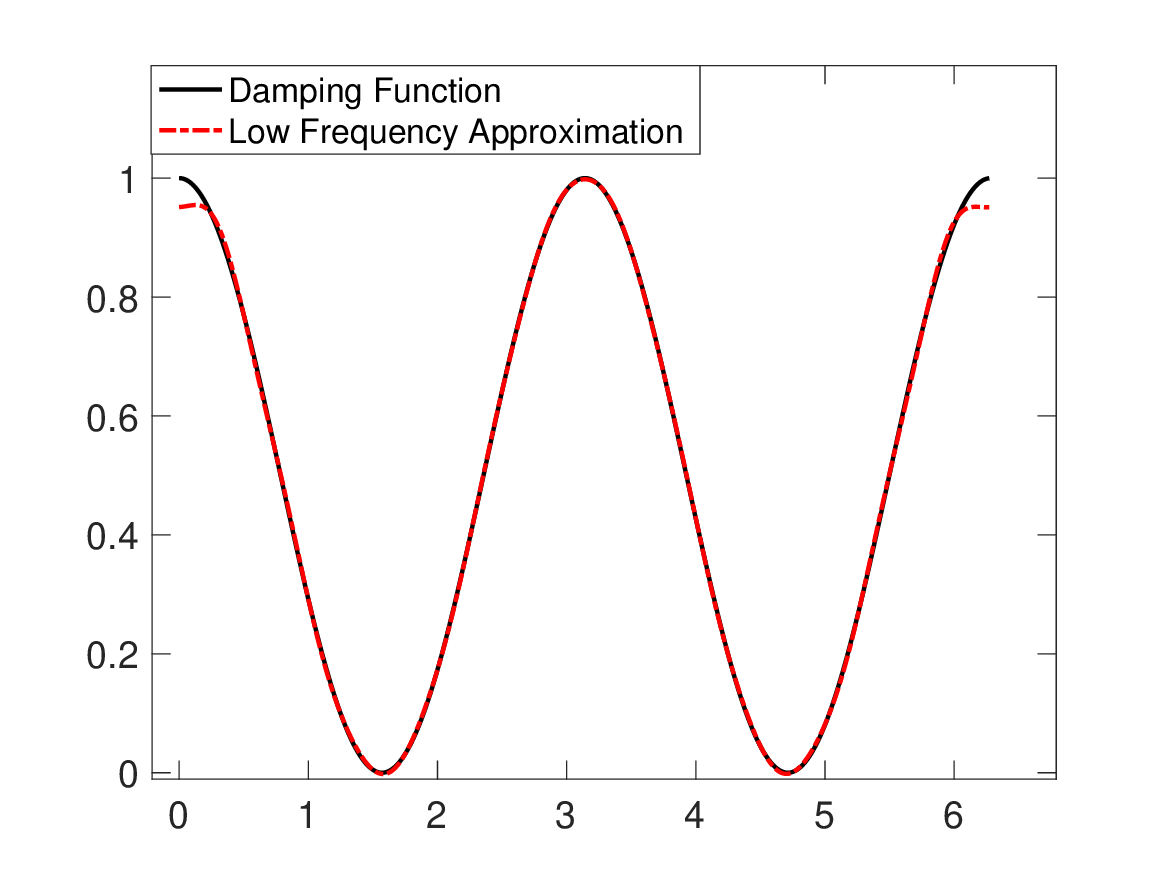} 
   \includegraphics[width=0.3\textwidth]{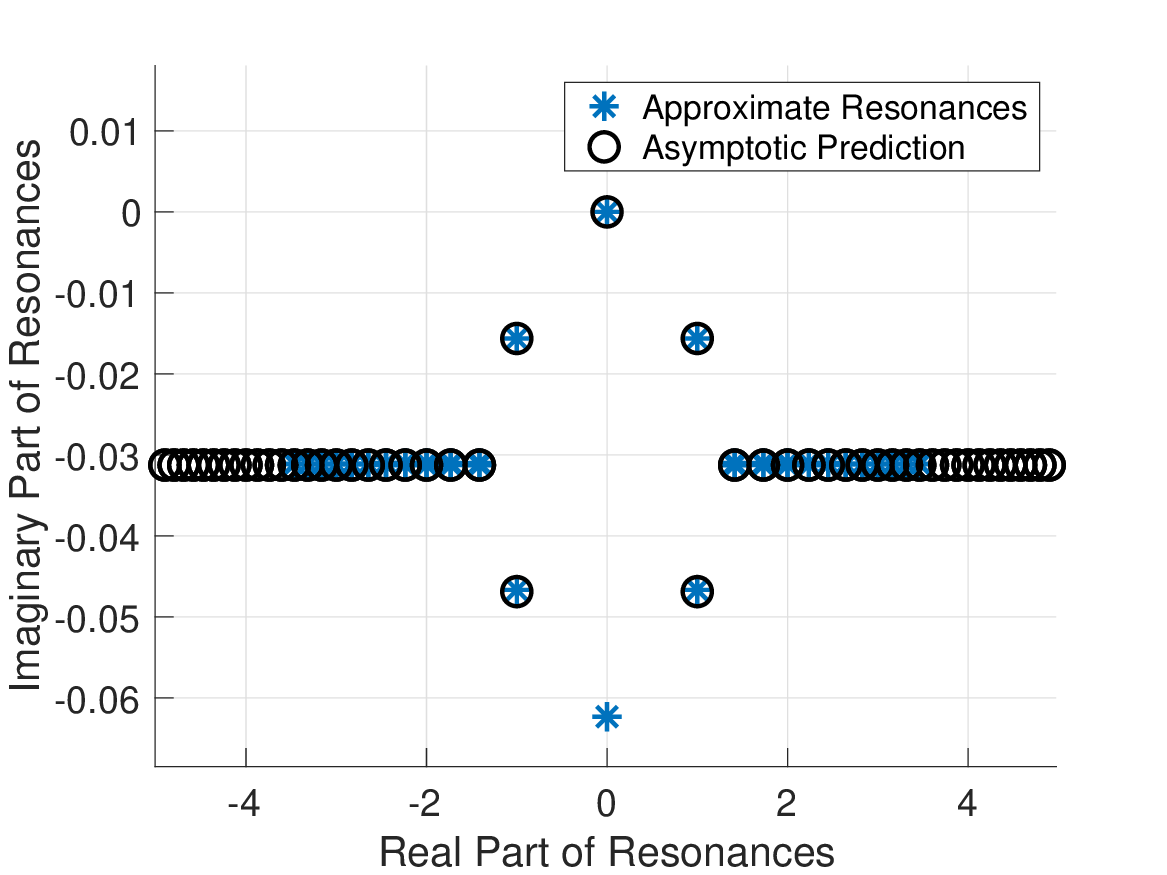} 
   \includegraphics[width=0.3\textwidth]{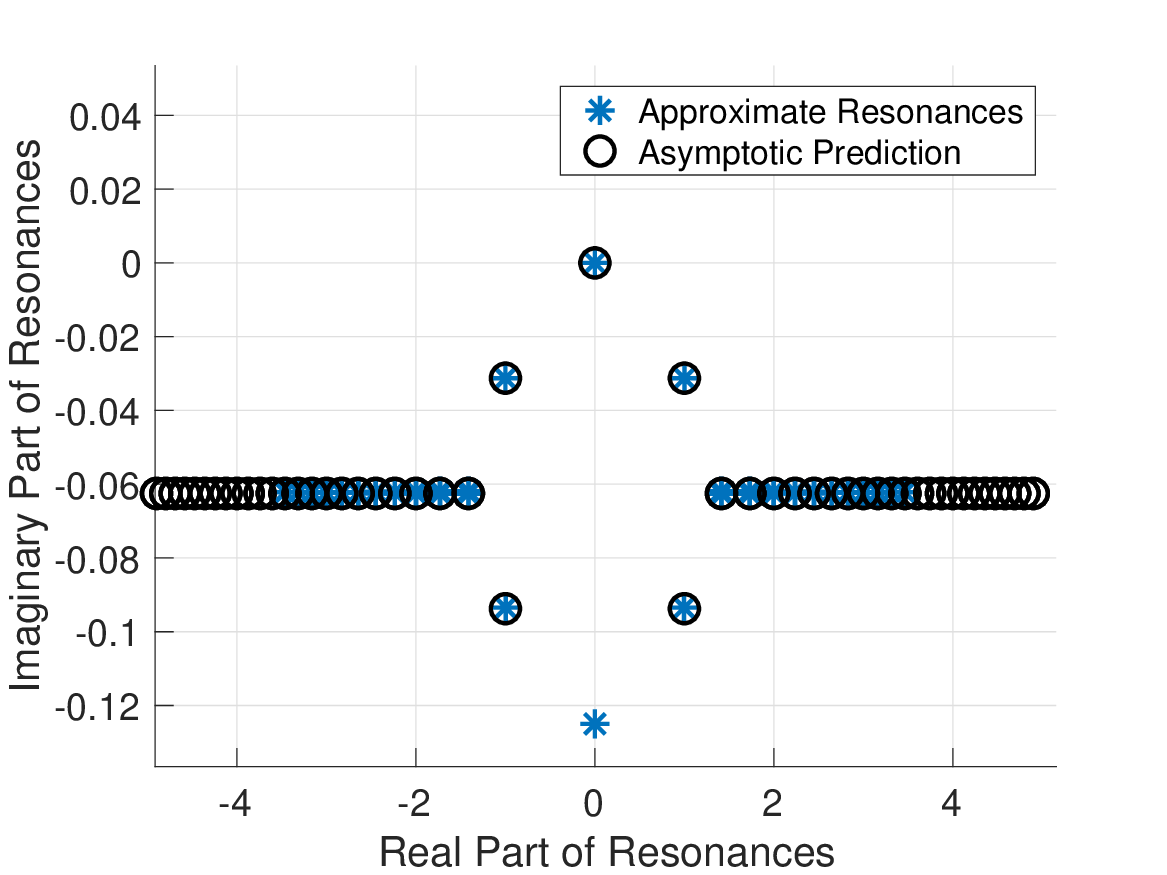} \\
      \includegraphics[width=0.3\textwidth]{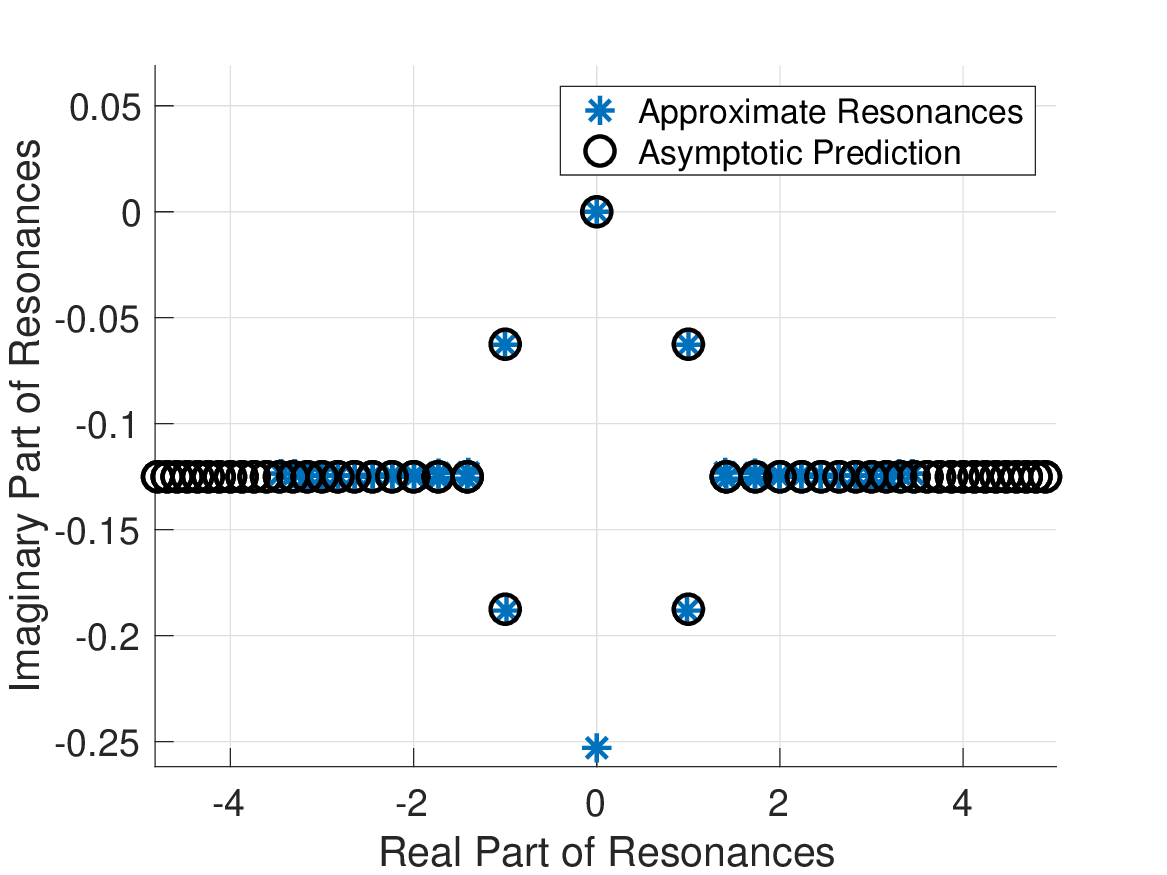} 
   \includegraphics[width=0.3\textwidth]{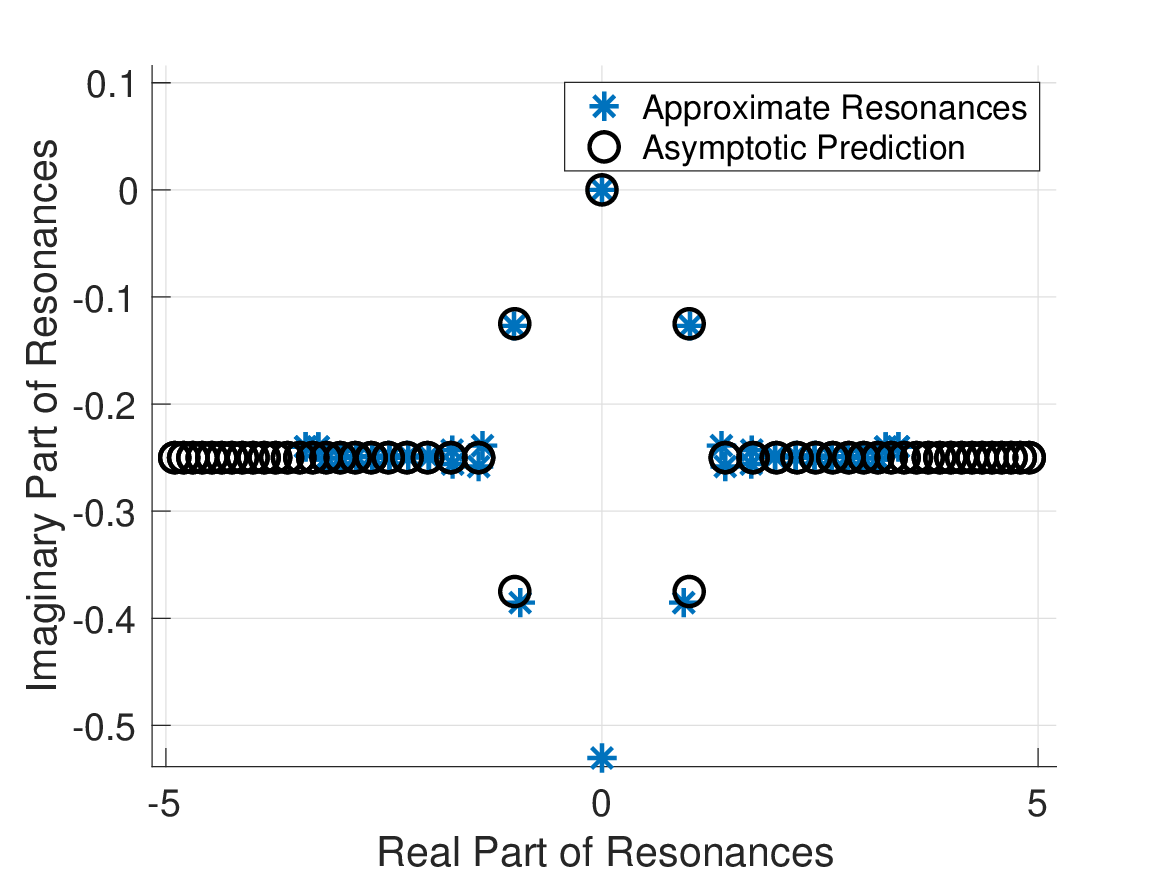}
   \caption{Clockwise from top left, a plot of the low-frequency approximation of the damping function compared to the original followed by the computed approximate resonances compared to prediction generated by $\chi_{1,\nu}$ for $\nu = .125, .25, .5, 1$ respectively.}
   \label{fig:res4}
\end{figure}

\begin{figure}[htbp] 
   \centering
      \includegraphics[width=0.3\textwidth]{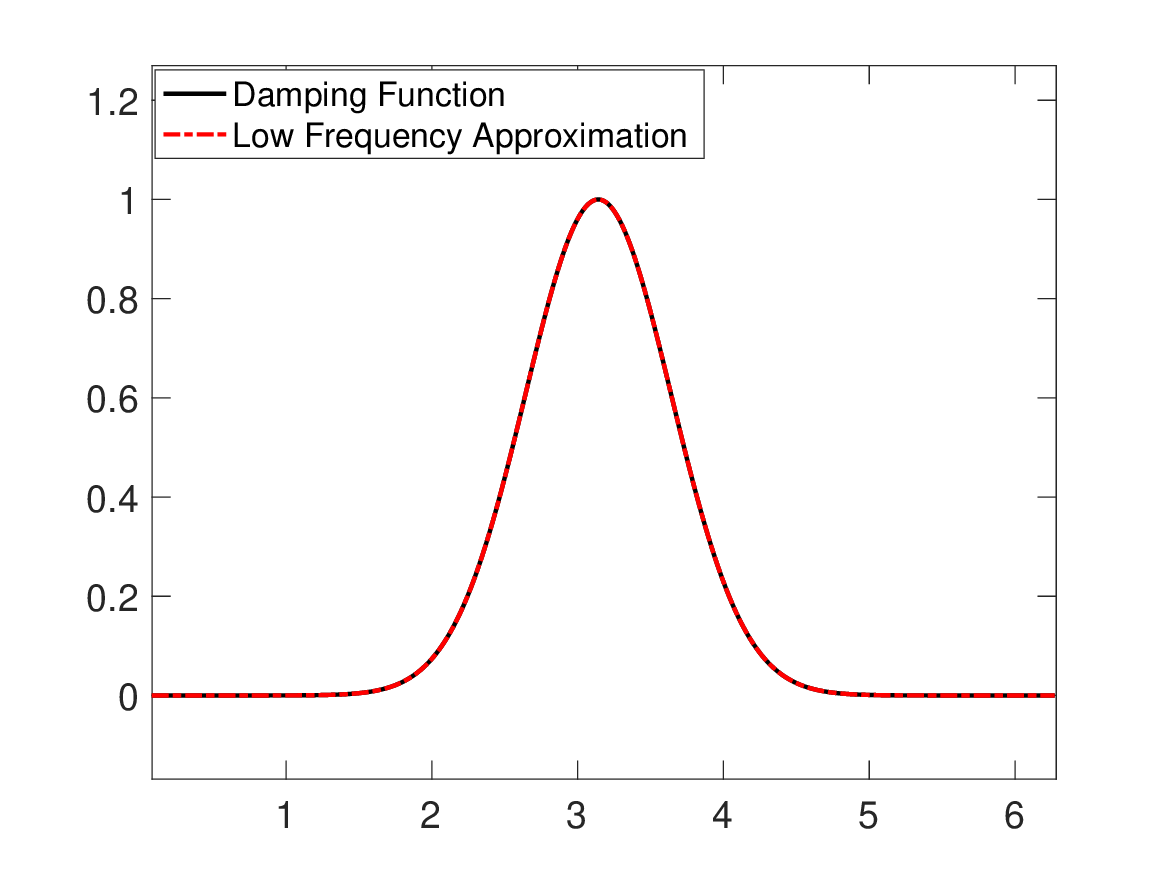} 
   \includegraphics[width=0.3\textwidth]{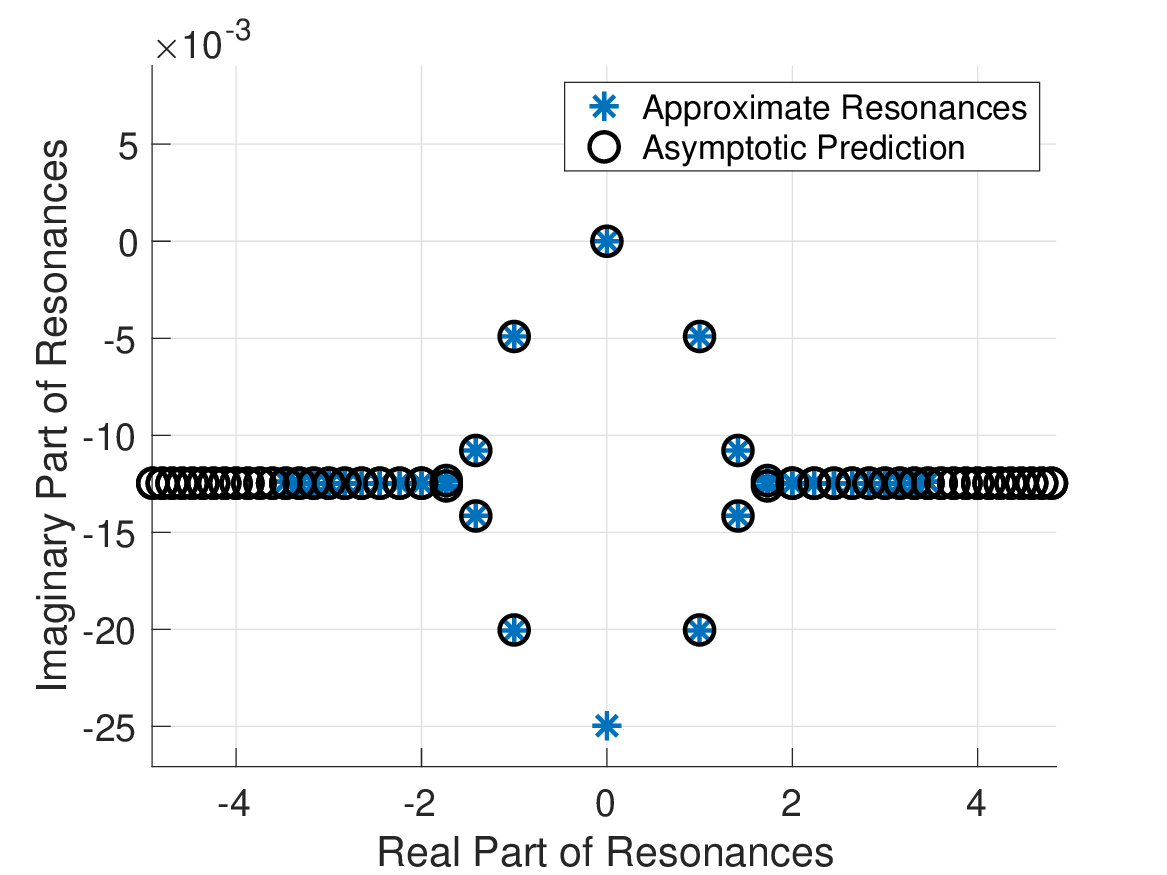} 
   \includegraphics[width=0.3\textwidth]{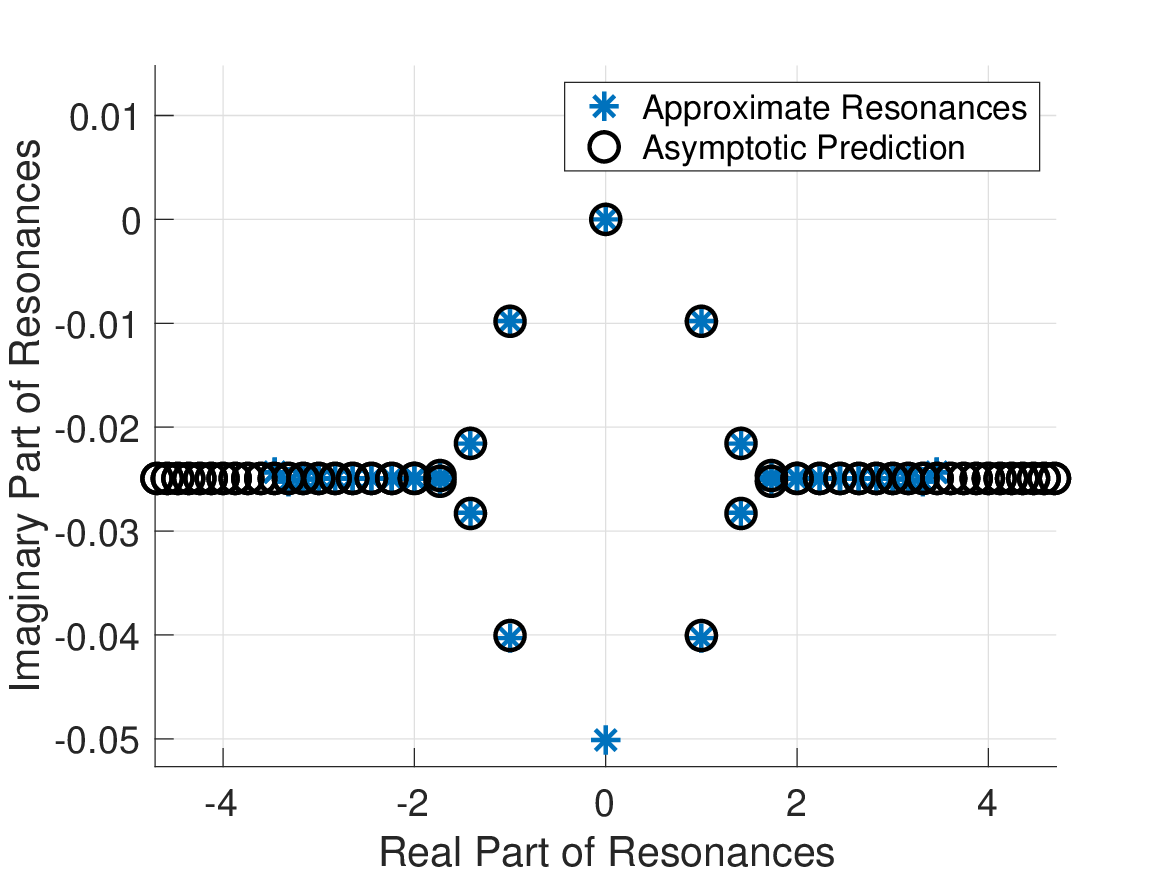}  \\
      \includegraphics[width=0.3\textwidth]{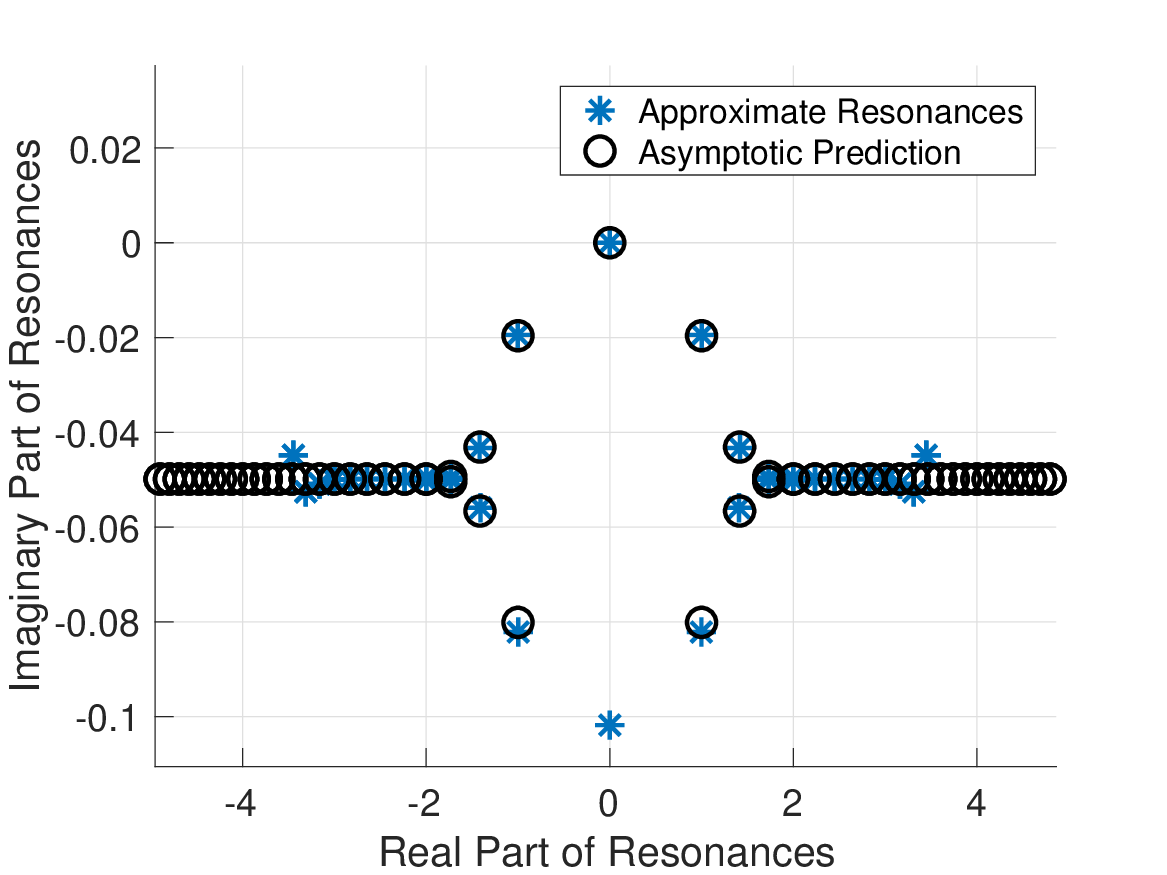} 
   \includegraphics[width=0.3\textwidth]{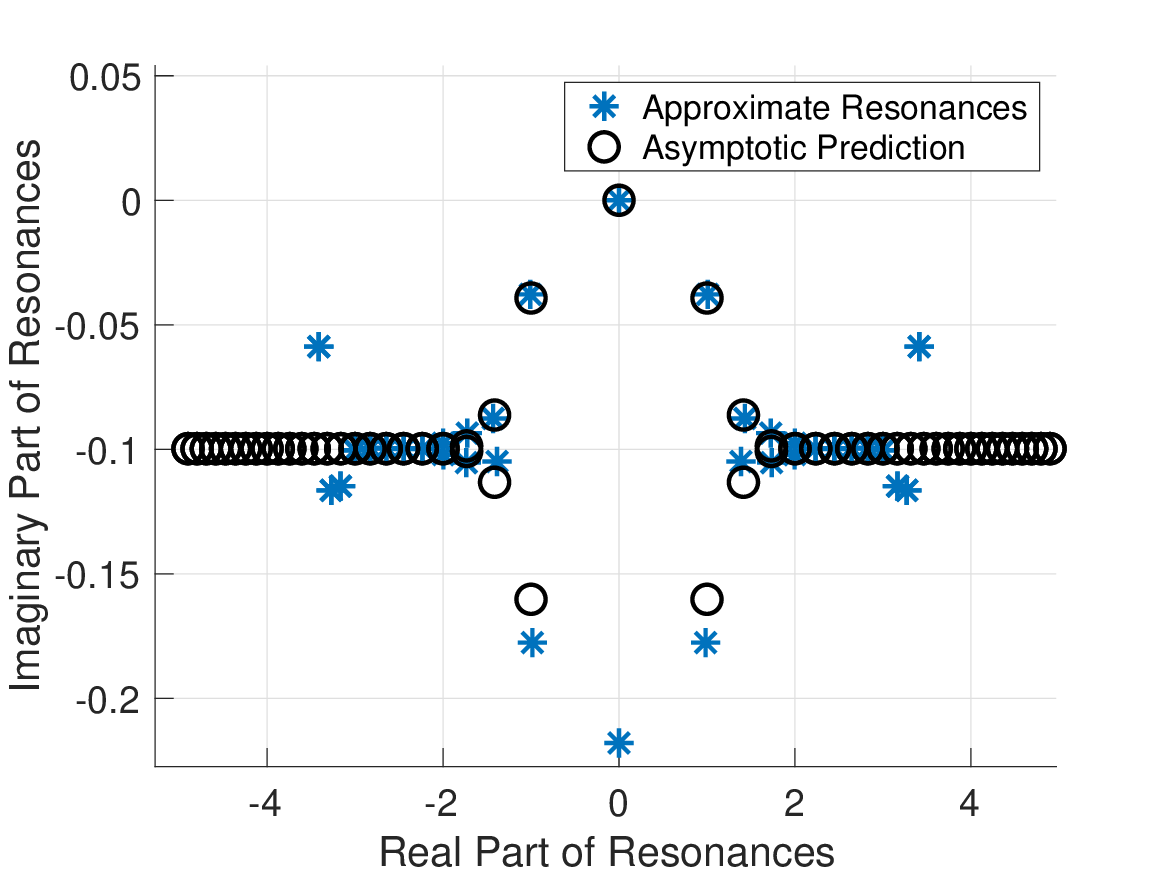}
   \caption{Clockwise from top left, a plot of the low-frequency approximation of the damping function compared to the original followed by the computed approximate resonances compared to prediction generated by $\chi_{2,\nu}$ for $\nu = .125, .25, .5, 1$ respectively.}
   \label{fig:res5}
\end{figure}

\begin{figure}[htbp] 
   \centering
      \includegraphics[width=0.3\textwidth]{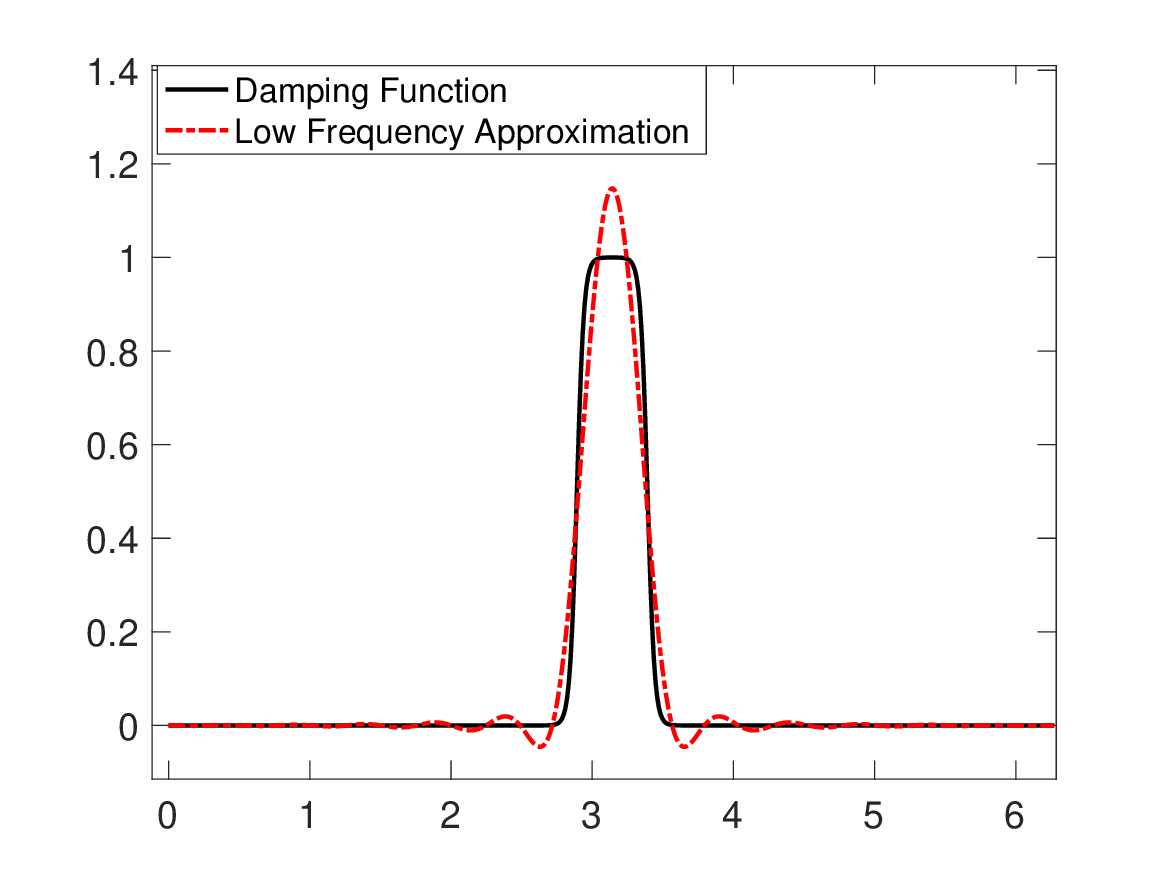} 
   \includegraphics[width=0.3\textwidth]{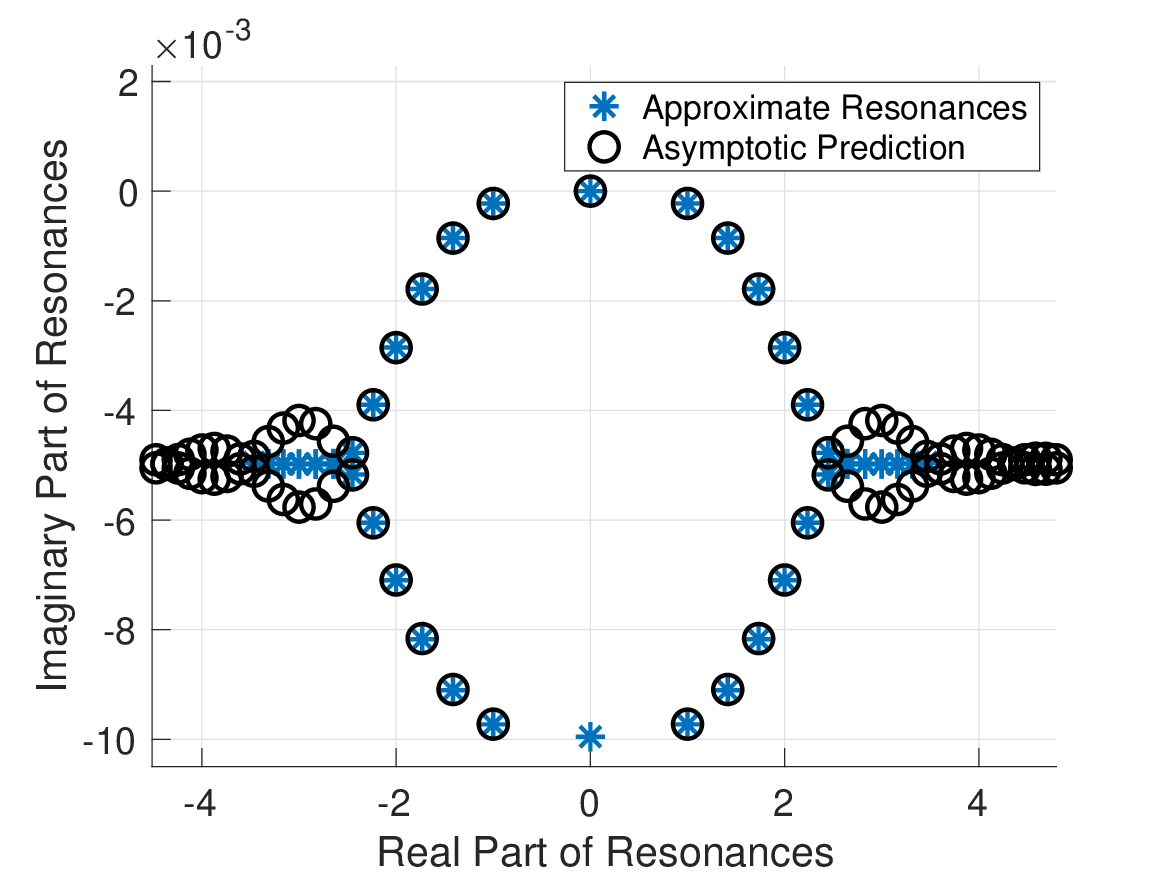} 
   \includegraphics[width=0.3\textwidth]{Approx_vs_Asymptotics_Resonances_Exp3_nu_Pt125.eps}  \\
      \includegraphics[width=0.3\textwidth]{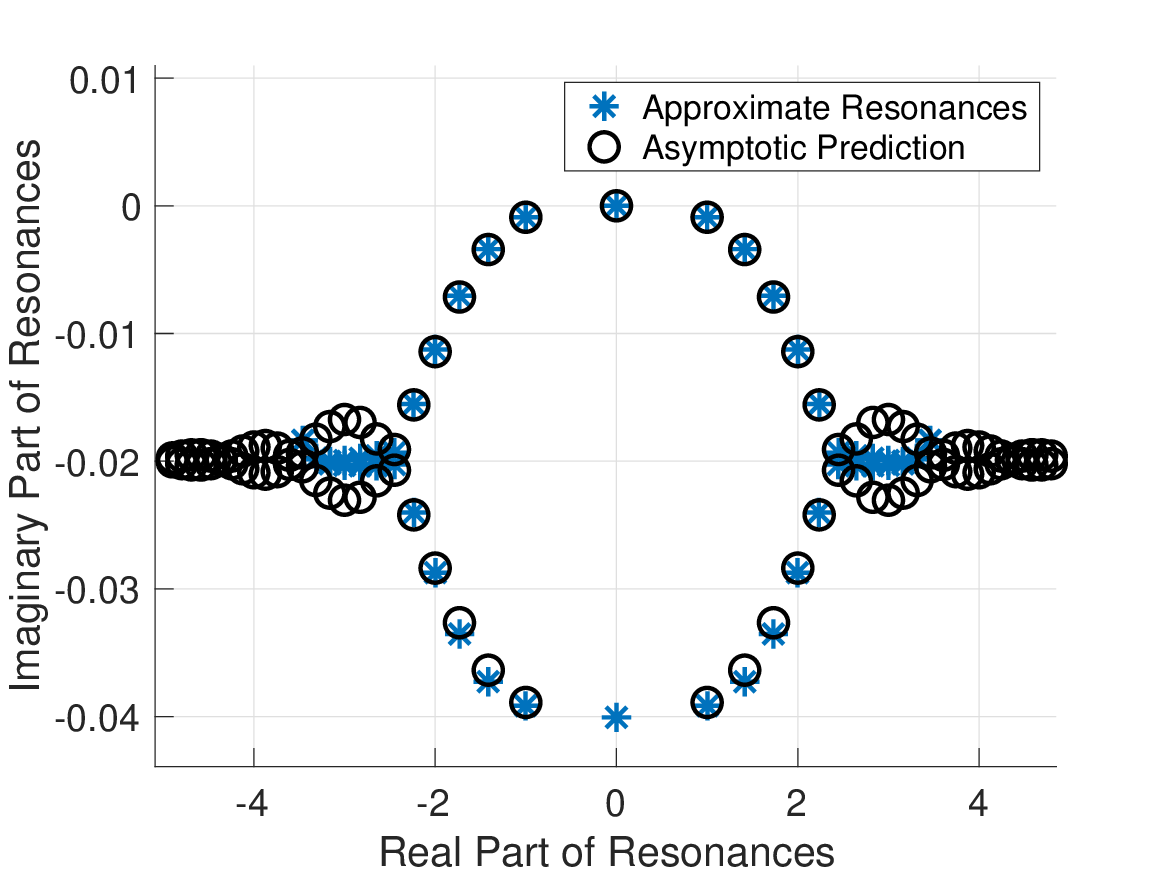} 
   \includegraphics[width=0.3\textwidth]{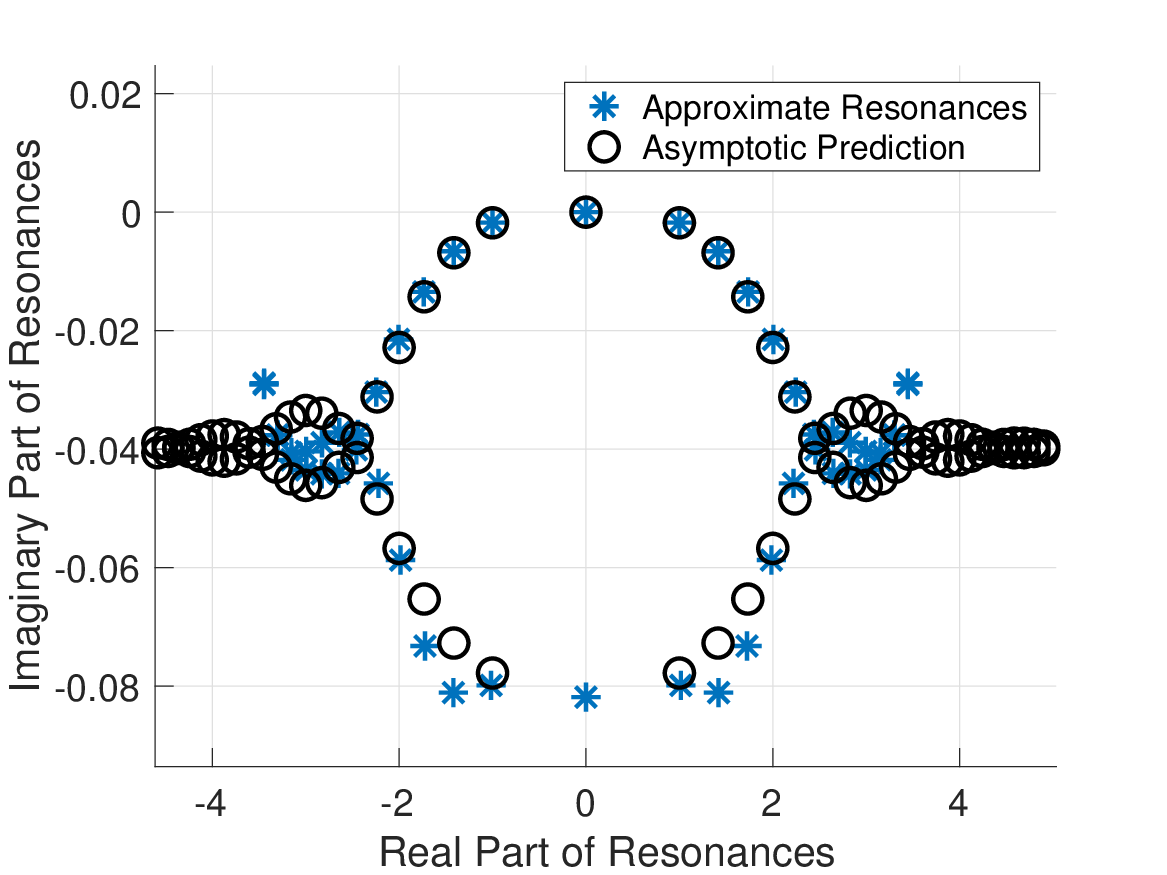}
   \caption{Clockwise from top left, a plot of the low-frequency approximation of the damping function compared to the original followed by the computed approximate resonances compared to prediction generated by $\chi_{3,\nu}$ for $\nu = .125, .25, .5, 1$ respectively.}
   \label{fig:res6}
\end{figure}

\subsection{Improved decay for higher regularity initial data}

Given our approximate resonance values, if we had a resonance free strip otherwise, we should have that $E(t)^{\frac12} \sim e^{\mathrm{Im} (z) t}$ where $E$ is defined as in \eqref{eq: def-energy}.  We have seen that this is not expected given our observed polynomial decay rates above.  However, given highly regular data, we can prove that our decay rate will be better than any polynomial as in the Remark at the end of \S \ref{ResToSG}.  However, we cannot provide such resonance free strips with our existing propagation estimates and based upon our estimates believe that no such strip exists at high energies.  Even so, we can test the evolution of \eqref{eqn:dampwave} for $\chi_{1,2,3}$ and compare the observed rate of convergence for $E^{\frac12}$ to those suggested by our computed resonances with very regular initial data.  These findings are displayed in Figure \ref{fig:res1}, where from left to right we plot the numerically observed decay rates of $E(t)^{\frac12}$ for each $\chi_j$, $j=1,2,3$.  Specifically, we plot $-\log (E(t))/2 t$ vs. $t$ and compare to the exponential decay rate we would expect from the low-energy resonances.  To ensure we observe slow decay rates in particular, in each of these simulations, the initial data is taken to be
\[
U (0,x) = \sin (x),
\]
which by our approximation theory is the function that dominates the low energy resonances we can compute that are closest to the real axis.  As suggested by our results, the decay rate does not fit an explicit asymptotic profile related to a specific resonance, though it does decay strongly and appears to decay at close to the expected exponential rate corresponding to the computed resonances at low energy.

\begin{figure}[htbp] 
   \centering
   \includegraphics[width=0.3\textwidth]{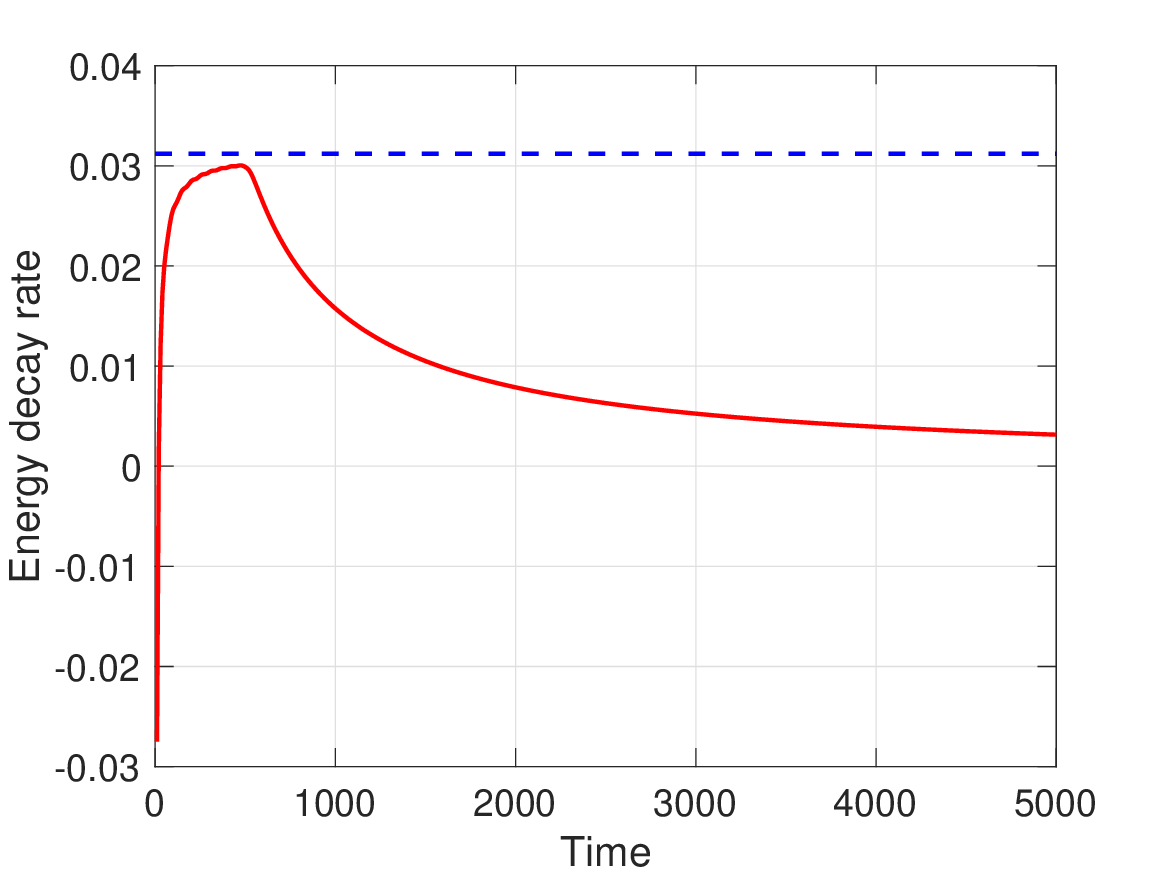} 
   \includegraphics[width=0.3\textwidth]{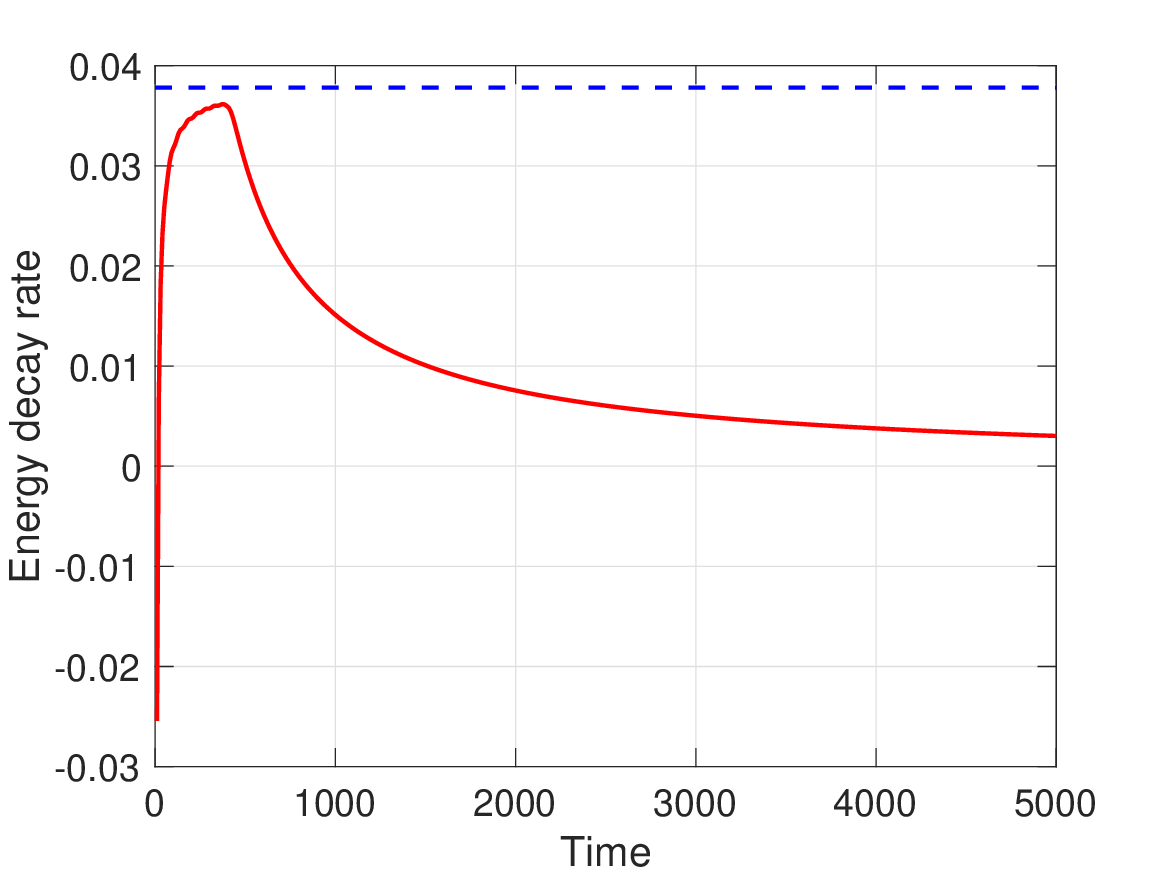}
   \includegraphics[width=0.3\textwidth]{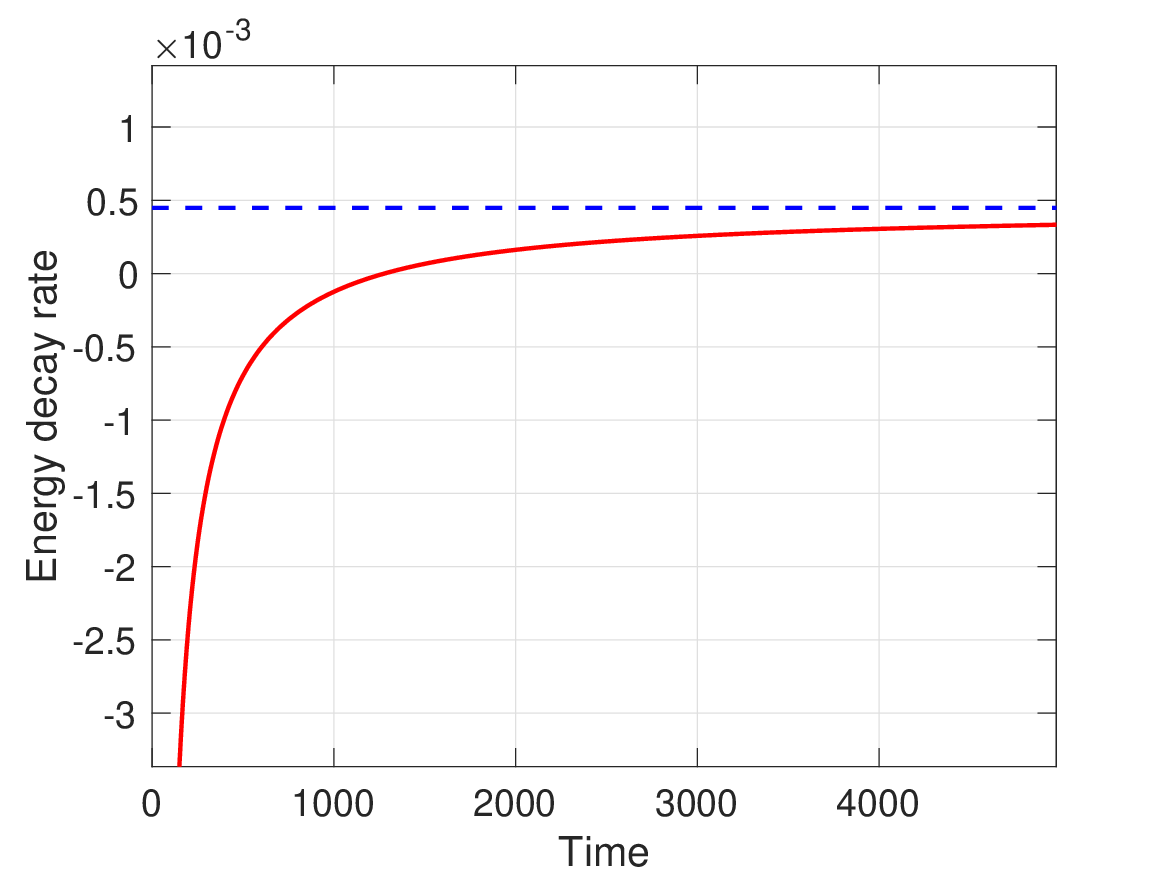}
   \caption{The numerically observed value of $-\log (E(t))/t$ for the solution to \eqref{eqn:dampwave} with $U(0,x) = \sin(x)$ (red) compared to the resonance prediction (blue) for \textbf{(Left)} $\chi_1$ damping. \textbf{(Middle)} $\chi_2$ damping. \textbf{(Right)} $\chi_3$ damping.}
   \label{fig:res1}
\end{figure}

\subsection{The full water wave problem}

To demonstrate how effective the damping we present here can be in the full model however, we also have included a model wave train from a forced-damped water wave model solved with very high precision using the techniques of \cite{Mar_Prior30}.  We illustrate this in Figure \ref{fig:water}.  As can be seen clearly, the damping appears to have a very strong local effect and allows for nonlinear wave trains to exist stably far from the damping.  

\begin{figure}[htbp] 
   \centering
   \includegraphics[width=0.6\textwidth]{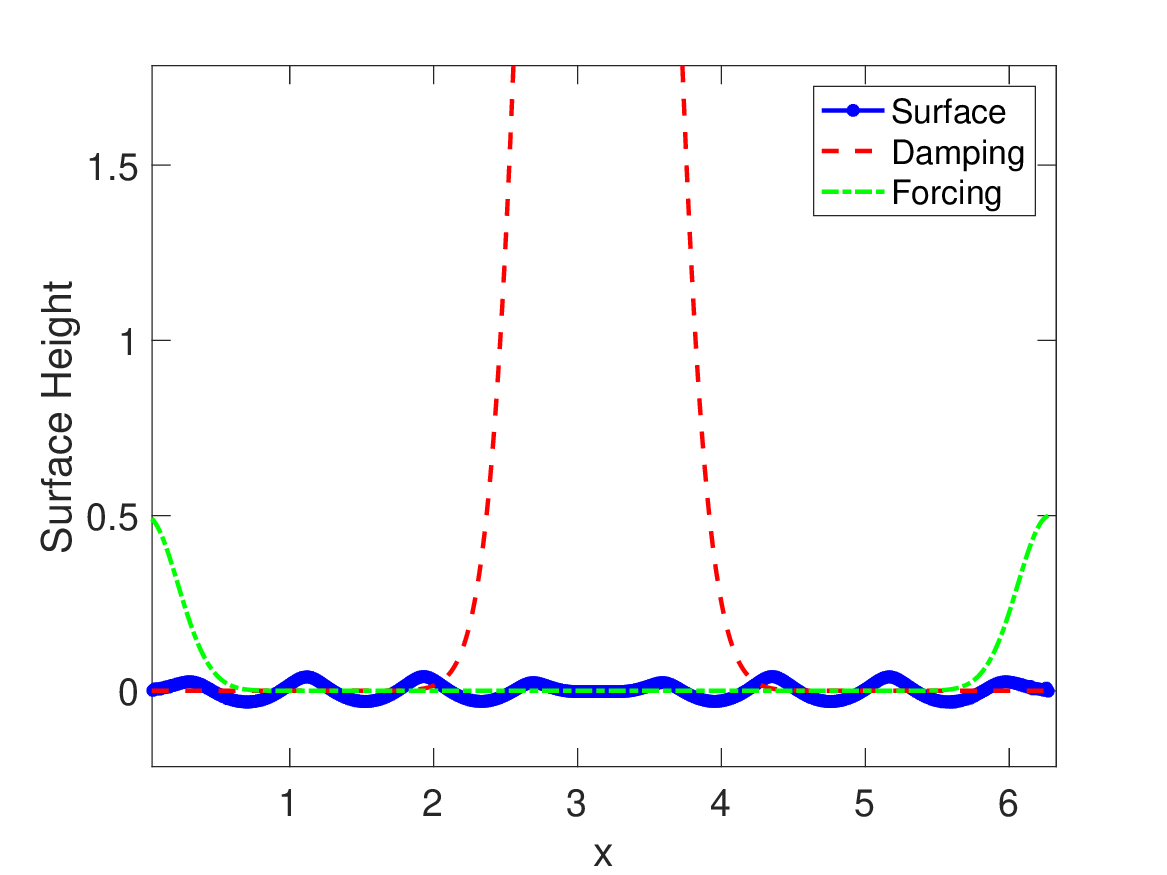} 
   \caption{A time slice of a forced-damped water wave train.}
   \label{fig:water}
\end{figure}

\bibliography{dampedwave}
\bibliographystyle{alpha}

\end{document}